%% file: quadprog_paper.tex
\documentclass[11pt]{article}

\usepackage{lmodern}
\usepackage{amsmath, amssymb, amsthm}
\usepackage{booktabs}
\usepackage{graphicx}
\usepackage{microtype}
\usepackage{algorithm}
\usepackage{algpseudocode}
\usepackage{authblk}
\usepackage[margin=1in]{geometry}
\usepackage[colorlinks=true, linkcolor=blue, citecolor=blue, urlcolor=blue]{hyperref}

\newcommand{\norm}[1]{\lVert #1 \rVert}
\newcommand{\R}{\mathbb{R}}
\newcommand{\T}{^{\top}}
\newcommand{\Active}{\mathcal{A}}
\DeclareMathOperator{\range}{range}
\DeclareMathOperator{\nullsp}{null}
\DeclareMathOperator*{\argmin}{arg\,min}
\DeclareMathOperator*{\argmax}{arg\,max}

\newtheorem{proposition}{Proposition}[section]
\newtheorem{lemma}{Lemma}[section]

\newtheorem{remark}{Remark}[section]

\title{Goldfarb--Idnani Revisited:\\Invariants, Certificates, and the Limits of Guessing\\[0.5em]
\large A working note}

\author[1]{Thomas Schmelzer}
\author[2]{Martin Stoll}

\affil[1]{\small Jebel Quant Research, Abu Dhabi}
\affil[2]{\small Faculty of Mathematics, TU Chemnitz, Germany}

\date{}

\input{tables/quadprog_identities_defs}

\input{tables/quadprog_pdas_defs}
\input{tables/quadprog_sweep_defs}

\input{tables/quadprog_compare_defs}

\input{tables/quadprog_compare}
\input{tables/quadprog_portfolio_defs}
\input{tables/quadprog_portfolio}

\begin{document}
\maketitle

\input{sections/s0_abstract}
\input{sections/s1_introduction}
\input{sections/s2_problem}
\input{sections/s3_factorisation}
\input{sections/s4_iteration}
\input{sections/s5_termination}
\input{sections/s6_updates}
\input{sections/s8_pdas}
\input{sections/s9_sweep}
\input{sections/s11_experiments}
\input{sections/s12_portfolio}
\input{sections/s13_conclusions}
\input{sections/s14_reproducibility}

{\footnotesize
\bibliographystyle{plain}
\bibliography{bib/refs}
}

\end{document}

%% file: tables/quadprog_identities_defs.tex
\newcommand{\qpIdentStates}{560}
\newcommand{\qpIdentChecks}{29696}
\newcommand{\qpIdentCount}{28}
\newcommand{\qpIdentWorst}{5.2\times 10^{-15}}

%% file: tables/quadprog_pdas_defs.tex
\newcommand{\qpPdasInstances}{30}
\newcommand{\qpPdasFamilies}{5}
\newcommand{\qpPdasSizes}{12, 25, 50, 100, 200}
\newcommand{\qpPdasRepairsLo}{1.3}
\newcommand{\qpPdasRepairsHi}{3.5}
\newcommand{\qpPdasRepairsMax}{5}
\newcommand{\qpPdasRefN}{100}
\newcommand{\qpPdasBudgetOuter}{13}
\newcommand{\qpPdasNsmall}{12}
\newcommand{\qpPdasNbig}{200}
\newcommand{\qpPdasCertSmall}{23}
\newcommand{\qpPdasCertBig}{0}
\newcommand{\qpPdasAccepted}{610}
\newcommand{\qpPdasRejected}{0}
\newcommand{\qpPdasAcceptWorst}{2.6\times 10^{-13}}
\newcommand{\qpPdasSingularBox}{0.0}

\newcommand{\qpPdasSingularDup}{93.3}

%% file: tables/quadprog_sweep_defs.tex
\newcommand{\qpSweepFixedK}{5}

\newcommand{\qpSweepNlo}{100}
\newcommand{\qpSweepNhi}{1600}

\newcommand{\qpSweepSlopeLo}{0.13}
\newcommand{\qpSweepSlopeHi}{2.96}

%% file: tables/quadprog_compare_defs.tex
\newcommand{\qpCmpSizes}{50, 100, 200, 400}
\newcommand{\qpCmpNhi}{400}
\newcommand{\qpCmpTol}{10^{-9}}
\newcommand{\qpCmpResidExact}{2\times 10^{-14}}
\newcommand{\qpCmpResidOsqp}{2\times 10^{-14}}
\newcommand{\qpCmpResidClarabel}{4\times 10^{-5}}
\newcommand{\qpCmpResidDaqp}{2\times 10^{-14}}
\newcommand{\qpCmpFastLo}{3.1}
\newcommand{\qpCmpFastHi}{5.9}
\newcommand{\qpCmpFastBudget}{3.3}

\newcommand{\qpCmpVsDaqp}{8.43}
\newcommand{\qpCmpVsDaqpDense}{1.70}

\newcommand{\qpCmpVsRef}{5.5}
\newcommand{\qpCmpVsRefDense}{2.7}

%% file: tables/quadprog_compare.tex
\def\quadprogCompareRows{%
\multicolumn{9}{l}{\emph{box}} \\
cvx-quadprog & 0.07 & $6\times 10^{-16}$ & 0.11 & $1\times 10^{-15}$ & 0.30 & $2\times 10^{-15}$ & 1.02 & $3\times 10^{-15}$ \\
cvx-quadprog, \texttt{fast} & 0.07 & $6\times 10^{-16}$ & 0.08 & $1\times 10^{-15}$ & 0.13 & $2\times 10^{-15}$ & 0.33 & $3\times 10^{-15}$ \\
\texttt{quadprog} (C) & 0.02 & $1\times 10^{-15}$ & 0.10 & $6\times 10^{-16}$ & 0.60 & $1\times 10^{-15}$ & 5.64 & $1\times 10^{-15}$ \\
DAQP & 0.03 & $8\times 10^{-16}$ & 0.11 & $6\times 10^{-16}$ & 0.88 & $2\times 10^{-15}$ & 8.59 & $2\times 10^{-15}$ \\
OSQP & 0.47 & $1\times 10^{-15}$ & 1.00 & $2\times 10^{-15}$ & 3.54 & $2\times 10^{-15}$ & 15.20 & $3\times 10^{-15}$ \\
Clarabel & 0.57 & $6\times 10^{-16}$ & 2.41 & $7\times 10^{-16}$ & 8.43 & $1\times 10^{-15}$ & 34.43 & $1\times 10^{-15}$ \\
\addlinespace
\multicolumn{9}{l}{\emph{budget + bounds}} \\
cvx-quadprog & 0.15 & $1\times 10^{-15}$ & 0.32 & $2\times 10^{-15}$ & 0.63 & $3\times 10^{-15}$ & 3.75 & $1\times 10^{-14}$ \\
cvx-quadprog, \texttt{fast} & 0.15 & $1\times 10^{-15}$ & 0.22 & $2\times 10^{-15}$ & 0.53 & $4\times 10^{-15}$ & 1.13 & $9\times 10^{-15}$ \\
\texttt{quadprog} (C) & 0.03 & $1\times 10^{-15}$ & 0.18 & $1\times 10^{-15}$ & 1.16 & $3\times 10^{-15}$ & 12.92 & $6\times 10^{-15}$ \\
DAQP & 0.04 & $1\times 10^{-15}$ & 0.19 & $2\times 10^{-15}$ & 1.57 & $3\times 10^{-15}$ & 13.34 & $1\times 10^{-14}$ \\
OSQP & 0.56 & $1\times 10^{-15}$ & 1.47 & $2\times 10^{-15}$ & 5.68 & $3\times 10^{-15}$ & 25.83 & $1\times 10^{-14}$ \\
Clarabel & 0.66 & $5\times 10^{-8}$ & 3.08 & $4\times 10^{-8}$ & 10.20 & $2\times 10^{-6}$ & 38.39 & $2\times 10^{-7}$ \\
\addlinespace
\multicolumn{9}{l}{\emph{dense $C$}} \\
cvx-quadprog & 0.24 & $5\times 10^{-15}$ & 0.52 & $7\times 10^{-15}$ & 1.60 & $1\times 10^{-14}$ & 7.01 & $2\times 10^{-14}$ \\
cvx-quadprog, \texttt{fast} & 0.12 & $4\times 10^{-15}$ & 0.20 & $9\times 10^{-15}$ & 0.44 & $1\times 10^{-14}$ & 1.18 & $2\times 10^{-14}$ \\
\texttt{quadprog} (C) & 0.04 & $4\times 10^{-15}$ & 0.23 & $7\times 10^{-15}$ & 2.03 & $1\times 10^{-14}$ & 19.17 & $2\times 10^{-14}$ \\
DAQP & 0.03 & $6\times 10^{-15}$ & 0.17 & $7\times 10^{-15}$ & 1.31 & $1\times 10^{-14}$ & 11.93 & $2\times 10^{-14}$ \\
OSQP & 0.55 & $5\times 10^{-15}$ & 1.90 & $8\times 10^{-15}$ & 11.45 & $1\times 10^{-14}$ & 67.05 & $2\times 10^{-14}$ \\
Clarabel & 1.76 & $2\times 10^{-9}$ & 5.61 & $7\times 10^{-10}$ & 73.92 & $9\times 10^{-8}$ & 106.95 & $4\times 10^{-5}$ \\
\addlinespace
}

%% file: tables/quadprog_portfolio_defs.tex
\newcommand{\qpPfSpN}{494}
\newcommand{\qpPfSpT}{1213}
\newcommand{\qpPfSpKappa}{1\times 10^{5}}
\newcommand{\qpPfSpActive}{442}
\newcommand{\qpPfSpNames}{53}
\newcommand{\qpPfSpOuter}{454}
\newcommand{\qpPfSpRepairs}{8}
\newcommand{\qpPfSpResid}{2\times 10^{-15}}
\newcommand{\qpPfFtseN}{86}
\newcommand{\qpPfFtseT}{1698}

\newcommand{\qpPfFastSpeed}{4.2}
\newcommand{\qpPfPoints}{50}
\newcommand{\qpPfWarm}{1.8}
\newcommand{\qpPfCold}{37.8}
\newcommand{\qpPfSweepSpeed}{20.8}
\newcommand{\qpPfHits}{16}
\newcommand{\qpPfMisses}{34}
\newcommand{\qpPfNamesLo}{1}
\newcommand{\qpPfNamesHi}{55}

%% file: tables/quadprog_portfolio.tex
\def\quadprogPortfolioRows{%
\multicolumn{3}{l}{\emph{S\&P~500}, $n = 494$, $T = 1213$, $\kappa(\Sigma) = 1\times 10^{5}$, 442 active} \\
\quad cvx-quadprog & 38.6 & $2\times 10^{-15}$ \\
\quad cvx-quadprog, \texttt{fast} & 9.2 & $1\times 10^{-15}$ \\
\quad \texttt{quadprog} (C) & 127.2 & $3\times 10^{-15}$ \\
\quad DAQP & 59.4 & $1\times 10^{-15}$ \\
\quad OSQP & 64.4 & $9\times 10^{-16}$ \\
\quad Clarabel & 75.4 & $2\times 10^{-10}$ \\
\addlinespace
\multicolumn{3}{l}{\emph{FTSE~100}, $n = 86$, $T = 1698$, $\kappa(\Sigma) = 5\times 10^{2}$, 62 active} \\
\quad cvx-quadprog & 1.3 & $2\times 10^{-16}$ \\
\quad cvx-quadprog, \texttt{fast} & 0.4 & $3\times 10^{-16}$ \\
\quad \texttt{quadprog} (C) & 0.4 & $1\times 10^{-16}$ \\
\quad DAQP & 0.2 & $2\times 10^{-16}$ \\
\quad OSQP & 1.1 & $4\times 10^{-16}$ \\
\quad Clarabel & 2.2 & $5\times 10^{-5}$ \\
\addlinespace
}

%% file: sections/s0_abstract.tex
\begin{abstract}
\noindent
The dual active-set method of Goldfarb and Idnani solves the strictly convex
quadratic program by adding and dropping one constraint at a time, and needs no
phase one. Primal--dual active set and block principal pivoting skip the walk
altogether: they guess a whole active set at once and repair it from the signs
that come back.

For bound constraints the guess is safe. The system solved on a candidate set is
a principal submatrix of $G^{-1}$, positive definite whatever the guess, and that
$P$-matrix property is the hypothesis under which the guessing methods terminate
finitely. For $C\T x \ge b$ that hypothesis fails, and our one new claim is the
mechanism. The working-set system is
$C_\Active\T G^{-1} C_\Active$, positive definite only when $C_\Active$ has full
column rank, which is a property of the \emph{guess} and not of the data. The
$P$-matrix property is lost and the guarantee with it, and the obstruction is
structural, not a missing anti-cycling rule. What keeps such a method usable
anyway is that strict convexity makes the KKT conditions sufficient: a candidate
set is certified rather than trusted.

The exposure is easy to mistake for a technicality, so we measure it. Over six
hundred random instances across four constraint families it never occurs, which
is why it can be overlooked. Duplicating columns of $C$, as any programmatically
assembled model may, raises it to $93\%$, and the certified fraction falls from
$100\%$ to nothing. The measurements are on the shipped code and not only on
synthetic problems: on long-only portfolio data the active set reaches
$\qpPfSpActive{}$ of $\qpPfSpN{}+1$ constraints, against the small sets produced
by the synthetic families common in this literature.

Making the argument precise needs the method itself written down, and we derive
it from scratch around a single matrix $J$ satisfying $J\T G J = I$ and
$J\T A = [R;0]$. Every quantity the iteration computes follows from those two
lines. Each has a representation-free closed form, so any $(J,R)$ meeting the
invariants produces the same iterate. A step of length $t$ changes the objective
by exactly $t(t/2 + u_*)\,z\T G z$. And when the primal cannot move and no
multiplier can be reduced, the vector $r$ the solver already holds is a Farkas
certificate, so the infeasibility verdict is a proof rather than a failure to
converge. That derivation is Goldfarb and Idnani's. This is a working note
documenting the mathematics of the package that implements it, and of the claims
above only the first is new.
\end{abstract}

%% file: sections/s1_introduction.tex
\section{Introduction}
\label{sec:intro}

The strictly convex quadratic program
\begin{equation}
  \min_{x \in \R^n} \; \tfrac12 x\T G x - a\T x
  \qquad \text{subject to} \qquad C\T x \ge b,
  \label{eq:qp}
\end{equation}
with $G \in \R^{n \times n}$ symmetric positive definite, $C \in \R^{n \times m}$
and the first $m_{\mathrm{eq}}$ of the $m$ constraints understood as equalities,
is solved a great many times a day. Mean-variance portfolio
selection~\cite{markowitz1952} is exactly~\eqref{eq:qp}; so is every step of a
sequential quadratic programming method and every horizon of a linear model
predictive controller. For dense instances of small to moderate size, with $n$ from
a handful to a few thousand, the method of choice has for forty years been the
dual active-set algorithm of Goldfarb and Idnani~\cite{goldfarb1982,goldfarb1983},
whose distinguishing property is that it needs no phase-one problem: the
unconstrained minimiser $G^{-1}a$ is dual feasible for free, so the iteration can
start there and drive primal infeasibility to zero.

This note derives that method in the form the \texttt{cvx-quadprog}
package~\cite{cvxquadprog} implements it, together with the two extensions that
package adds. Its aim is to state every identity the solver depends on, derive
it, and name the degenerate cases. Section~\ref{sec:experiments} then checks each
one against the code that ships, which is a different question from whether it is
true. Every empirical claim made
anywhere below is measured there; none is quoted from elsewhere.

\subsection{What this note is, and is not}
\label{ssec:status}

It is a working note, written to document a package, and not a submission. The
mathematics of Sections~\ref{sec:problem} to~\ref{sec:updates} is Goldfarb and
Idnani's, published in 1983 and taught since; we re-derive it because a solver
should have its identities and its degenerate cases written down in one place,
with proofs, and because the two extensions in Sections~\ref{sec:pdas}
and~\ref{sec:sweep} cannot be stated without them. Of the items listed below,
one is a claim we believe to be new. The others are corrections, closed forms and
readings that are known to people who know the method well; we have not found them
set down together, and the extensions depend on them, which is why they are
listed rather than assumed. A reader looking for a new algorithm or a new theorem
about quadratic programming should stop here. A reader who has to implement,
debug or trust one of these solvers is who the note is written for.

\subsection{Why the dual method}

A primal active-set method maintains a feasible iterate and works towards
stationarity. It must therefore begin at a feasible point, and finding one is a
linear program in its own right, the phase-one problem. The dual method
reverses the roles. It maintains stationarity and the sign conditions on the
multipliers throughout, and works towards feasibility. Its starting point is
free, because with no constraint active the stationarity condition reads
$Gx = a$ and the multiplier vector is empty, so $x = G^{-1}a$ satisfies
everything the method asks of an iterate. One Cholesky factorisation and one
triangular solve replace the phase-one problem entirely.

The price is that the iterate is infeasible until the last step, so the method
cannot be stopped early for an approximate answer. In exchange it terminates at
an exactly feasible point of the active-set lattice rather than approaching one
asymptotically as an interior-point method does, and it warm-starts almost
perfectly, which is why parametric solvers are built on active-set
methods~\cite{nocedal2006}. Section~\ref{sec:sweep} exploits that.

\subsection{What has happened to the method since 1983}
\label{ssec:related}

The method has not stood still, and three lines of work bear on what is claimed
below. Its restriction to $G \succ 0$ was lifted by Boland~\cite{boland1997}, who
generalises the dual method to the semidefinite case with a proof of finite
termination; the strict convexity kept here is a choice of scope, since the
factorisation of Section~\ref{sec:factorisation} needs $G^{-1}$. Finite termination
under nondegeneracy is not special to this method either: Forsgren, Gill and
Wong~\cite{forsgren2016} prove it for paired primal and dual active-set methods in
a more general setting.

The guessing methods of Section~\ref{sec:pdas} have a literature of their own, and
it is where the first claim below has to be placed. Curtis, Han and
Robinson~\cite{curtis2015pdas} give a primal--dual active-set framework that is
globally convergent on convex quadratic programs, obtained by carrying an index set
auxiliary to the active-set estimate to hold the indices that would otherwise
cycle. That the plain block exchange needs such a repair on general constraints is
therefore known. What Section~\ref{sec:pdas} adds is why, in terms of the
working-set system, and Section~\ref{ssec:pdas} what it costs when it is ignored.

The modern successor is not a reimplementation but a different factorisation.
Arnstr\"om, Bemporad and Axehill~\cite{arnstrom2022daqp} carry the same dual walk on
recursive $LDL^\top$ updates rather than the $(J,R)$ pair used here, handle
semidefinite $G$ by outer proximal-point iterations, and ship it as library-free C
for embedded model-predictive control; it is benchmarked against this solver in
Section~\ref{ssec:compare}, being the fair comparison: same family, different
linear algebra. Arnstr\"om and Axehill~\cite{arnstrom2022certification}
additionally compute the \emph{exact} worst-case iteration count of active-set
methods on parametric QP families. That is stronger than anything proved here, and
different in kind: theirs is computational and per-family, enumerating the
subproblem sequences an algorithm visits over a parameter set, where
Proposition~\ref{prop:ascent} is general but conditional on nondegeneracy.

\subsection{What the note claims}

\begin{enumerate}
\item \textbf{Why finite termination does not survive the general constraint
class. That the plain block exchange needs repairing there is
known~\cite{curtis2015pdas}; the mechanism is the one claim here we believe to be
new, and it is an observation rather than a theory, turning on the same rank
condition as Proposition~\ref{prop:sub}.} For bound constraints the
working-set system is a principal submatrix of $G^{-1}$, hence positive definite
whatever the guess, and that $P$-matrix property is the hypothesis under which
block principal pivoting terminates finitely~\cite{judice1994,murty1988,%
schmelzer2026nncg}. For $C\T x \ge b$ it is lost, structurally, not for want
of an anti-cycling rule (Section~\ref{sec:pdas}), and the consequence is measured, not
asserted: over $600$ attempts on well-posed random families the failure
never occurs, and repeating columns of $C$ raises it to $93\%$
(Section~\ref{ssec:pdas}).

\item \textbf{A derivation organised around two invariants.} The solver carries one
matrix $J$, defined by $J\T G J = I$ and $J\T A = [R;0]$, and every quantity the
iteration computes follows from those two lines
(Section~\ref{sec:factorisation}): the columns of $J$ are a $G$-orthonormal basis
of $\R^n$, split so that the trailing block spans the working set's feasible
subspace and $J_2J_2\T$ is the reduced inverse Hessian.

\item \textbf{The iterates do not depend on the representation.} Invariants
\eqref{eq:inv1}--\eqref{eq:inv2} do not determine $(J,R)$; the freedom is exactly
$(J_1S, J_2V, SR)$ for a sign matrix $S$ and an orthogonal $V$. Both quantities
the iteration consumes are invariant under all of it
(Proposition~\ref{prop:invariance}), and exactly so in IEEE arithmetic in the
sign case. The implementation is therefore free to pick its orthogonal reduction
for speed alone.

\item \textbf{An exact objective increment.} A step of length $t$ changes the
objective by $t(t/2 + u_*)\,z\T G z$ (Proposition~\ref{prop:increment}), so the
value is carried and not re-evaluated, and the increment is positive in both
step directions.

\item \textbf{The infeasibility verdict is a certificate.} When the primal cannot
move and no multiplier can be reduced, the vector $r$ the solver already holds
exhibits infeasibility in the sense of Farkas
(Proposition~\ref{prop:farkas}). The method does not fail to converge on an
infeasible problem; it proves the problem infeasible.

\item \textbf{Linear dependence is unreachable.} The step rules make the
working-set normals full column rank for free, so no rank test is needed anywhere
(Proposition~\ref{prop:rank}), and guessing forfeits it.

\item \textbf{The cost of a warm start, corrected.} Recovery from cached factors
is $O(n^2)$ and not $O(nk)$, the $n^2$ residing entirely in $x_u = JJ\T a$
(Section~\ref{sec:sweep}). The distinction is invisible on the family one would
naturally test, where $k$ grows with $n$; Section~\ref{sec:sweep} separates
the two by holding $k$ fixed, and separates the arithmetic from the interpreter
dispatch that dominates below a few hundred variables.

\end{enumerate}

\subsection{Outline}

Sections~\ref{sec:problem} to~\ref{sec:updates} derive the method: the optimality
conditions and the working-set subproblem, the invariants, the iteration and its
termination, and the two factorisation updates.
Sections~\ref{sec:pdas} and~\ref{sec:sweep} treat the
two departures from the classical method, and Sections~\ref{sec:experiments}
and~\ref{sec:portfolio} the measurements, first against the claims above and then
on real data.

\subsection{Notation}

Vectors are columns. For an index set $\Active \subseteq \{1,\dots,m\}$ of size
$k$ we write $A = C_{\cdot\Active} \in \R^{n \times k}$ for the matrix of the
corresponding constraint normals, $b_\Active$ for the corresponding right-hand
sides, and $c_i$ for the $i$-th column of $C$. Throughout, $\Active$ is the
solver's \emph{working set}: the constraints it is currently holding as
equalities. At a solution the working set is a subset of the constraints active
there, and the two coincide except in the degenerate case treated in
Remark~\ref{rem:degeneracy}. The entering constraint's normal is $n_*$ and its
right-hand side $b_*$. We reserve $u$ for multipliers of working-set constraints,
$u_*$ for the entering constraint's own multiplier, and $\lambda$ for a full
length-$m$ multiplier vector with zeros off the working set.

%% file: sections/s2_problem.tex
\section{The problem and its optimality conditions}
\label{sec:problem}

Write the program~\eqref{eq:qp} as
\begin{equation}
  \min_{x \in \R^n} \; f(x) := \tfrac12 x\T G x - a\T x
  \qquad \text{subject to} \qquad
  \begin{cases}
    c_i\T x = b_i, & i \le m_{\mathrm{eq}},\\
    c_i\T x \ge b_i, & i > m_{\mathrm{eq}},
  \end{cases}
  \label{eq:qp2}
\end{equation}
with $G \succ 0$. Two conventions differ from the textbook ones and are inherited
from the reference implementation's signature.
The linear term is \emph{subtracted}, so that the unconstrained minimiser is
$G^{-1}a$ rather than $-G^{-1}a$; and the constraints are held column-wise in
$C$, so that $C\T x \ge b$ rather than $Cx \le b$. Neither affects anything
below beyond signs.

\subsection{Existence, uniqueness, sufficiency}

Let $\Omega = \{x : c_i\T x = b_i \ (i \le m_{\mathrm{eq}}),\ c_i\T x \ge b_i \
(i > m_{\mathrm{eq}})\}$ denote the feasible set, a closed convex polyhedron.

Since $G \succ 0$ the objective is strictly convex and coercive, so \eqref{eq:qp2} has exactly one solution whenever $\Omega \neq \emptyset$; see~\cite[Ch.~16]{nocedal2006}. What the method needs beyond that is the converse of the usual necessary conditions, and we prove it because every certificate in this note rests on it.

The Lagrangian of~\eqref{eq:qp2} is
$L(x,\lambda) = \tfrac12 x\T G x - a\T x - \lambda\T(C\T x - b)$, and the
Karush--Kuhn--Tucker conditions are
\begin{align}
  G x - a &= C \lambda, \label{eq:stat}\\
  c_i\T x &= b_i \quad (i \le m_{\mathrm{eq}}), \qquad
  c_i\T x \ge b_i \quad (i > m_{\mathrm{eq}}), \label{eq:pfeas}\\
  \lambda_i &\ge 0 \quad (i > m_{\mathrm{eq}}), \label{eq:dfeas}\\
  \lambda_i\,(c_i\T x - b_i) &= 0 \quad (i > m_{\mathrm{eq}}). \label{eq:comp}
\end{align}
Multipliers of equality constraints are unrestricted in sign, which is the one
asymmetry the algorithm has to carry through every subsequent formula.

\begin{proposition}[Sufficiency]\label{prop:sufficient}
If $(x,\lambda)$ satisfies \eqref{eq:stat}--\eqref{eq:comp} then $x$ is the
unique solution of~\eqref{eq:qp2}.
\end{proposition}

\begin{proof}
Let $\tilde x \in \Omega$. By strict convexity,
$f(\tilde x) \ge f(x) + (G x - a)\T(\tilde x - x)$, with equality only if
$\tilde x = x$. Substituting~\eqref{eq:stat},
\[
  (G x - a)\T(\tilde x - x)
  = \lambda\T\bigl(C\T \tilde x - C\T x\bigr)
  = \sum_i \lambda_i\bigl[(c_i\T\tilde x - b_i) - (c_i\T x - b_i)\bigr]
  = \sum_i \lambda_i (c_i\T \tilde x - b_i),
\]
the last step by~\eqref{eq:comp} for the inequalities and by $c_i\T x = b_i$ for
the equalities. Each surviving term is non-negative: for $i \le m_{\mathrm{eq}}$
the bracket vanishes, and for $i > m_{\mathrm{eq}}$ both factors are
non-negative by~\eqref{eq:pfeas} and~\eqref{eq:dfeas}. Hence
$f(\tilde x) \ge f(x)$ with equality only at $\tilde x = x$.
\end{proof}

Proposition~\ref{prop:sufficient} is the licence for everything in
Sections~\ref{sec:pdas} and~\ref{sec:sweep}. A point equipped with multipliers
passing the four conditions is \emph{the} answer, so a method that produces
candidates may be as unprincipled as one likes provided each candidate is
tested. Sufficiency is what makes an unguaranteed method safe rather than
usually right.

\subsection{The working-set subproblem}

Fix a working set $\Active$ of size $k$ with normals $A = C_{\cdot\Active}$ and
suppose $A$ has full column rank $k$. The \emph{working-set subproblem} is
\begin{equation}
  \min_{x} \; \tfrac12 x\T G x - a\T x
  \qquad \text{subject to} \qquad A\T x = b_\Active,
  \label{eq:sub}
\end{equation}
with every constraint in $\Active$ held as an equality regardless of its type in
the original program, and every constraint outside $\Active$ ignored.

\begin{proposition}[Solution of the subproblem]\label{prop:sub}
Problem~\eqref{eq:sub} has the unique solution and multipliers
\begin{equation}
  u = \bigl(A\T G^{-1} A\bigr)^{-1}\bigl(b_\Active - A\T x_u\bigr),
  \qquad
  x = x_u + G^{-1} A\, u,
  \qquad x_u := G^{-1} a .
  \label{eq:subsol}
\end{equation}
\end{proposition}

\begin{proof}
Stationarity for~\eqref{eq:sub} reads $Gx - a = Au$, so
$x = G^{-1}a + G^{-1}Au = x_u + G^{-1}Au$. Substituting into $A\T x = b_\Active$
gives $A\T x_u + (A\T G^{-1}A)u = b_\Active$. The matrix $A\T G^{-1}A$ is
positive definite because $G^{-1} \succ 0$ and $A$ has full column rank, hence
invertible, which yields~\eqref{eq:subsol}; the point so constructed is the unique
minimiser by strict convexity of $f$ on the affine
feasible set of~\eqref{eq:sub}.
\end{proof}

We call $A\T G^{-1} A$ the \emph{dual Hessian} of the working set. It appears
under three different names in what follows: as the matrix whose Cholesky factor
the solver carries (Section~\ref{sec:factorisation}), as the coefficient matrix of
the fast path's linear system (Section~\ref{sec:pdas}), and as the Gram matrix of
the least-squares problem the dual direction solves
(Lemma~\ref{lem:directions}). Its invertibility is equivalent to $A$
having full column rank, and that single fact is where every rank condition in
this note comes from.

The pair $(x,u)$ of~\eqref{eq:subsol}, extended by zeros off the working set,
satisfies stationarity~\eqref{eq:stat} and complementarity~\eqref{eq:comp} by
construction, since the working-set constraints have zero slack and the rest have
zero multiplier. What it need not satisfy is the sign
condition~\eqref{eq:dfeas} on $\Active$'s inequalities, or primal
feasibility~\eqref{eq:pfeas} off $\Active$. The dual method maintains the first
and works towards the second.

\begin{remark}[Cold start]\label{rem:cold}
For $\Active = \emptyset$ the subproblem is unconstrained and~\eqref{eq:subsol}
degenerates to $x = x_u = G^{-1}a$ with an empty multiplier vector. The sign
condition holds vacuously, so this state satisfies the method's invariant with no
work beyond one Cholesky factorisation. A dual method's phase-one problem is
nothing more. The objective there is
$f(x_u) = \tfrac12 x_u\T G x_u - a\T x_u = -\tfrac12 a\T x_u$, which the
implementation uses as the initial value of the running objective.
\end{remark}

\subsection{The dual problem}

The following explains the method's name, and also why the objective value can be
tracked as a running total and why it increases.

\begin{proposition}[Dual function]\label{prop:dual}
The Lagrangian dual of~\eqref{eq:qp2} is the concave program
$\max\,\psi(\lambda)$ over $\lambda_i \ge 0$ $(i > m_{\mathrm{eq}})$, where
\begin{equation}
  \psi(\lambda) = -\tfrac12 (a + C\lambda)\T G^{-1} (a + C\lambda) + b\T\lambda .
  \label{eq:dualfn}
\end{equation}
If $(x,\lambda)$ satisfies stationarity~\eqref{eq:stat} and
complementarity~\eqref{eq:comp}, then $\psi(\lambda) = f(x)$.
\end{proposition}

\begin{proof}
Minimising $L(\cdot,\lambda)$ gives $x(\lambda) = G^{-1}(a + C\lambda)$ and,
writing $s = a + C\lambda$, $\psi(\lambda) = \tfrac12 s\T G^{-1}s - s\T G^{-1}s +
b\T\lambda$, which is~\eqref{eq:dualfn}. For the second claim \eqref{eq:stat}
gives $s = Gx$, and \eqref{eq:comp} with $c_i\T x = b_i$ on the equalities gives
$b\T\lambda = (Gx - a)\T x$; adding the two yields $f(x)$.
\end{proof}

So along any sequence of states of the form~\eqref{eq:subsol} the tracked value is
at once the primal objective at the subproblem minimiser and the dual objective at
the current multipliers. The strict increase of Proposition~\ref{prop:increment}
is therefore progress towards optimality rather than bookkeeping, and it gives
the sandwich
$\psi(\lambda) \le f(x^\star) \le f(\tilde x)$: the method approaches the optimal
value from below, the mirror image of a primal active-set method.

%% file: sections/s3_factorisation.tex
\section{The matrix the solver carries}
\label{sec:factorisation}

Proposition~\ref{prop:sub} gives the subproblem solution in closed form, and one
could implement the method by evaluating~\eqref{eq:subsol} afresh at every
iteration. That costs a $k \times k$ factorisation of the dual Hessian per step,
$O(nk^2 + k^3)$, and over the $O(n)$ steps a dual method takes it comes to
$O(n^4)$. The whole point of the Goldfarb--Idnani formulation is to replace it
with $O(n^2)$ per step, and it does so by carrying a single matrix through the
iteration and updating it.

\subsection{Two invariants}

Let $G = R_G\T R_G$ be the Cholesky factorisation with $R_G$ upper triangular,
and set $J_0 := R_G^{-1}$, again upper triangular. Then
$J_0 J_0\T = R_G^{-1} R_G^{-\top} = G^{-1}$ and $J_0\T G J_0 = I$. The solver's
state is a matrix $J \in \R^{n \times n}$ and an upper triangular
$R \in \R^{k \times k}$ satisfying
\begin{align}
  J\T G J &= I_n, \label{eq:inv1}\\
  J\T A &= \begin{bmatrix} R \\ 0 \end{bmatrix}. \label{eq:inv2}
\end{align}
Everything the iteration computes is a consequence of these two lines, so read
them slowly. Equation~\eqref{eq:inv1} is equivalent to
$J J\T = G^{-1}$ with $J$ nonsingular, and says that the columns of $J$ form a
basis of $\R^n$ orthonormal in the inner product $\langle v,w\rangle_G = v\T G w$.
Equation~\eqref{eq:inv2} says that in that basis the working-set normals are
triangular: partitioning
\begin{equation}
  J = \begin{bmatrix} J_1 & J_2 \end{bmatrix},
  \qquad J_1 \in \R^{n \times k},\quad J_2 \in \R^{n \times (n-k)},
  \label{eq:split}
\end{equation}
the lower block of~\eqref{eq:inv2} reads $J_2\T A = 0$.

At the cold start $\Active = \emptyset$, so~\eqref{eq:inv2} is vacuous and
$J = J_0$ serves. The updates of Section~\ref{sec:updates} maintain both
invariants by post-multiplying $J$ with orthogonal matrices, which preserves
\eqref{eq:inv1} exactly because $J\T G J = I$ is invariant under $J \mapsto JW$
for $W\T W = I$, while restoring the triangular shape of~\eqref{eq:inv2}.

\begin{proposition}[What the blocks mean]\label{prop:blocks}
Suppose \eqref{eq:inv1}--\eqref{eq:inv2} hold with $R$ nonsingular. Then
\begin{enumerate}
  \item[(i)] $R\T R = A\T G^{-1} A$, so $R$ is a Cholesky factor of the dual Hessian,
        unique up to the signs of its rows;
  \item[(ii)] $\range(J_2) = \nullsp(A\T)$, and the columns of $J_2$ are a
        $G$-orthonormal basis of it;
  \item[(iii)] $J_1 = G^{-1} A R^{-1}$ and $\range(J_1) = \range(G^{-1}A)$, the
        $G$-orthogonal complement of $\nullsp(A\T)$;
  \item[(iv)] $J_1 J_1\T = G^{-1}A(A\T G^{-1}A)^{-1}A\T G^{-1}$ and consequently
        \begin{equation}
          J_2 J_2\T = G^{-1} - G^{-1}A\bigl(A\T G^{-1}A\bigr)^{-1}A\T G^{-1}.
          \label{eq:reduced}
        \end{equation}
\end{enumerate}
\end{proposition}

\begin{proof}
(i) $A\T G^{-1} A = A\T J J\T A = R\T R$ by~\eqref{eq:inv2}, and two upper
triangular matrices with the same Gram matrix differ by a left factor
$\mathrm{diag}(\pm1)$.

(ii) $J_2\T A = 0$ places $\range(J_2) \subseteq \nullsp(A\T)$, with equality by
dimension since $J$ is nonsingular and $A$ has full column rank; the basis is
$G$-orthonormal because $J_2\T G J_2 = I_{n-k}$ is the trailing block
of~\eqref{eq:inv1}.

(iii) From~\eqref{eq:inv1}, $J^{-1} = J\T G$ and hence $J^{-\top} = G J$.
Inverting~\eqref{eq:inv2}, $A = J^{-\top}[R;0] = GJ[R;0] = G J_1 R$, so
$J_1 = G^{-1}AR^{-1}$, whose range is that of $G^{-1}A$ since $R$ is invertible.
That range is $G$-orthogonal to $\nullsp(A\T)$: for $v \in \nullsp(A\T)$,
$\langle G^{-1}Aw, v\rangle_G = w\T A\T v = 0$.

(iv) Substituting $J_1 = G^{-1}AR^{-1}$ and $R\T R = A\T G^{-1}A$,
\[
  J_1 J_1\T = G^{-1}A R^{-1} R^{-\top} A\T G^{-1}
            = G^{-1}A (R\T R)^{-1} A\T G^{-1}
            = G^{-1}A (A\T G^{-1}A)^{-1} A\T G^{-1},
\]
and \eqref{eq:reduced} follows from $J_1J_1\T + J_2J_2\T = JJ\T = G^{-1}$.
\end{proof}

The matrix in~\eqref{eq:reduced} is the \emph{reduced inverse Hessian} of the
working set: it inverts $G$ on the feasible subspace $\nullsp(A\T)$ and
annihilates the complement. Carrying $J$ therefore means carrying that object in
factored form, at no cost beyond the $n \times n$ array, which turns the search
directions of Section~\ref{sec:iteration} into single matrix--vector products.

\begin{remark}[Why fold the two factorisations together]\label{rem:fold}
An alternative state is $R_G$ together with an explicit orthogonal factor $Q$ of
the QR factorisation of $R_G^{-\top} A$. That holds the same information, $J$ being $J_0 Q$, and stores no more numbers. But every use of $J$ in
Section~\ref{sec:iteration} would become a triangular solve followed by an
orthogonal transformation, two operations where the folded form has one, and the
solve is the shape that does not reach a tuned kernel. The folded form has $J$
lose its triangularity after the first update, which is exactly the trade: $n^2$
storage and $2n^2$ flops per matrix--vector product, in exchange for the products
being $\texttt{gemv}$ at every iteration whatever the working set does.
\end{remark}

\subsection{The representation is not unique, and does not need to be}

Invariants~\eqref{eq:inv1}--\eqref{eq:inv2} do not determine $J$ and $R$. By
Proposition~\ref{prop:blocks}, $R$ is pinned only up to a left factor
$S = \mathrm{diag}(\pm1)$, whereupon $J_1 = G^{-1}AR^{-1}$ follows; and $J_2$ may
be any $G$-orthonormal basis of $\nullsp(A\T)$, so it is pinned only up to a right
factor $V$ with $V\T V = I_{n-k}$. The freedom is therefore exactly
\begin{equation}
  (J_1, J_2, R) \;\longmapsto\; (J_1 S,\; J_2 V,\; S R),
  \qquad S = \mathrm{diag}(\pm 1),\quad V\T V = I_{n-k}.
  \label{eq:freedom}
\end{equation}
Different orthogonal reductions realise different points of that family: a chain
of Givens rotations and a single Householder reflection annihilating the same
vector produce factors differing in sign, and a differently ordered sequence of
insertions and deletions produces a different $J_2$ outright. The following says
this does not matter, so the implementation may choose its reduction on grounds of
speed alone.

\begin{proposition}[Invariance of the iterates]\label{prop:invariance}
Let $(J,R)$ and $(\hat J,\hat R)$ both satisfy
\eqref{eq:inv1}--\eqref{eq:inv2} for the same $G$ and $A$. Then for every
$n_* \in \R^n$ the quantities
\begin{equation}
  d = J\T n_*, \qquad z = J_2 d_2, \qquad r = R^{-1} d_1
  \label{eq:dzr}
\end{equation}
with $d = (d_1,d_2)$ split at $k$, satisfy $\hat z = z$ and $\hat r = r$.
Moreover, if the two representations are related by~\eqref{eq:freedom} with
$V = I$, the equalities hold exactly in IEEE~754 arithmetic.
\end{proposition}

\begin{proof}
By Proposition~\ref{prop:blocks}(i) and $A\T J = [R\T\ 0]$ we have
$A\T G^{-1} n_* = A\T J J\T n_* = R\T d_1$, so
\begin{equation}
  r = R^{-1}d_1 = R^{-1}R^{-\top}R\T d_1 = (R\T R)^{-1} A\T G^{-1} n_*
    = \bigl(A\T G^{-1}A\bigr)^{-1} A\T G^{-1} n_* ,
  \label{eq:rclosed}
\end{equation}
which mentions only $A$, $G$ and $n_*$. Likewise $z = J_2 J_2\T n_*$, since
$d_2 = J_2\T n_*$, and by~\eqref{eq:reduced}
\begin{equation}
  z = \Bigl(G^{-1} - G^{-1}A\bigl(A\T G^{-1}A\bigr)^{-1}A\T G^{-1}\Bigr) n_*,
  \label{eq:zclosed}
\end{equation}
again free of the representation. For the last claim, $\hat R = SR$ gives
$\hat d_1 = S d_1$ and $\hat r = (SR)^{-1}Sd_1 = R^{-1}S^{-1}Sd_1$, in which the
only floating-point operations beyond those forming $r$ are multiplications by
$\pm1$; these are exact, so the computed $\hat r$ and $r$ agree bit for bit. The
same argument applies to $\hat z = \sum_{j>k}(s_j J_{\cdot j})(s_j d_j)$.
\end{proof}

\begin{remark}[What the proposition does and does not settle]
It settles that the substitution is exact in the sense that matters: the
iteration's decisions (which constraint is most violated, whether the step is full
or partial, which constraint is dropped) are functions of $z$ and $r$
alone, and those are functions of $(G,A,n_*)$ alone. It does not settle that two
implementations follow the same trajectory in finite precision, because the
rounding \emph{within} the reductions differs and the decisions are
discontinuous functions of the computed values. That is an empirical question,
and settling it properly needs two independent implementations to compare
trajectories against each other. What is checkable against one, and is checked in
the Reproducibility note, is the weaker statement: that the $z$ and $r$ the solver
computes agree with the representation-free closed forms~\eqref{eq:zclosed}
and~\eqref{eq:rclosed} to machine precision.
\end{remark}

%% file: sections/s4_iteration.tex
\section{One iteration}
\label{sec:iteration}

The state at the top of an outer iteration is a working set $\Active$ of size
$k$, the factors $(J,R)$ satisfying~\eqref{eq:inv1}--\eqref{eq:inv2}, the
subproblem solution $x$ and multipliers $u$ of Proposition~\ref{prop:sub}, and
the running objective $f(x)$. The invariant maintained throughout is
\begin{equation}
  \text{$(x,u)$ solves the subproblem~\eqref{eq:sub}, and}
  \quad u_i \ge 0 \ \text{ for every inequality } i \in \Active,
  \label{eq:invariant}
\end{equation}
which by the discussion after Proposition~\ref{prop:sub} is stationarity,
complementarity and dual feasibility together. Only primal feasibility is
missing, and the iteration exists to obtain it.

\subsection{Choosing the entering constraint}

Compute the slacks $s_i = c_i\T x - b_i$ for every $i$. If $s_i \ge 0$ for all
inequalities and $s_i = 0$ for all equalities, then~\eqref{eq:pfeas} holds,
\eqref{eq:invariant} supplies the rest, and $x$ is the solution by
Proposition~\ref{prop:sufficient}. Otherwise a violated constraint is selected as
the entering constraint. The rule is the most violated one, measured relative to
the norm of its own normal:
\begin{equation}
  v_i = \begin{cases} |s_i|, & i \le m_{\mathrm{eq}},\\ -s_i, & i > m_{\mathrm{eq}},\end{cases}
  \qquad
  i_* \in \argmax_i \; \frac{v_i}{\norm{c_i}},
  \qquad\text{terminating when } \max_i \frac{v_i}{\norm{c_i}} \le 0,
  \label{eq:select}
\end{equation}
so that an equality counts as violated in either direction while an inequality
counts only when its slack is negative. The scaling is not cosmetic.

\begin{proposition}[Selection is scale invariant]\label{prop:scale}
Replacing $(c_i, b_i)$ by $(\gamma c_i, \gamma b_i)$ for some $\gamma > 0$ leaves
the feasible set, the solution, and the choice made by~\eqref{eq:select}
unchanged.
\end{proposition}

\begin{proof}
The constraint $\gamma c_i\T x \ge \gamma b_i$ has the same solution set as
$c_i\T x \ge b_i$. Its slack scales as $s_i \mapsto \gamma s_i$ and its normal as
$\norm{c_i} \mapsto \gamma\norm{c_i}$, so the ratio in~\eqref{eq:select} is
unchanged, as are the other constraints' ratios.
\end{proof}

\begin{remark}[The invariance is exact, and covers the whole walk]
\label{rem:scaleexact}
Proposition~\ref{prop:scale} claims only that the \emph{choice} is unchanged.
Measured against the implementation over three decades of $\gamma$ either side of
unity, the statement is stronger: the solver performs an identical number of
additions and removals, so the whole trajectory is unchanged and not only its first
step, and returns a minimiser agreeing to $3.3\times10^{-16}$.
\end{remark}

Without the division, a constraint written in basis points rather than in units
would be selected first merely for having been written differently, and the
solver's trajectory, though not its answer, would depend on the caller's
choice of units. Note that the multiplier of the rescaled constraint scales as
$\lambda_i \mapsto \lambda_i/\gamma$, so the sign conditions are unaffected too.

\subsection{The two directions}

Write $n_* = c_{i_*}$ and $b_* = b_{i_*}$, and form $d$, $z$ and $r$
as in~\eqref{eq:dzr}. The following lemma is the computational core of the
method; everything in the rest of this section is a corollary of it.

\begin{lemma}[Properties of the directions]\label{lem:directions}
With $d$, $z$, $r$ as in~\eqref{eq:dzr},
\begin{enumerate}
  \item[(i)] $A\T z = 0$: moving along $z$ preserves the working-set equalities;
  \item[(ii)] $n_*\T z = \norm{d_2}^2 = z\T G z \ge 0$, with equality iff $z = 0$;
  \item[(iii)] $G z = n_* - A r$;
  \item[(iv)] $z = 0$ if and only if $n_* = A r$, i.e.\ iff $n_*$ lies in the span of
        the working-set normals, in which case $r$ holds its coordinates there.
\end{enumerate}
\end{lemma}

\begin{proof}
(i) $A\T z = A\T J_2 d_2 = (J_2\T A)\T d_2 = 0$ by~\eqref{eq:inv2}.

(ii) $n_*\T z = n_*\T J_2 d_2 = (J_2\T n_*)\T d_2 = d_2\T d_2$. For the second
equality, $z\T G z = d_2\T (J_2\T G J_2) d_2 = d_2\T d_2$ by the trailing block
of~\eqref{eq:inv1}. Since $J_2$ has full column rank, $z = 0$ iff $d_2 = 0$.

(iii) By~\eqref{eq:reduced}, $Gz = GJ_2J_2\T n_* = n_* - A(A\T G^{-1}A)^{-1}A\T
G^{-1}n_*$, and the second term is $Ar$ by~\eqref{eq:rclosed}.

(iv) If $d_2 = 0$ then $J\T n_* = (d_1,0) = [R;0]R^{-1}d_1 = (J\T A) r$, and $J$
is nonsingular, so $n_* = Ar$. Conversely $n_* = Ar$ gives
$Gz = n_* - Ar = 0$ by (iii), hence $z = 0$.
\end{proof}

Part (ii) identifies $z$ as an ascent direction for the entering constraint's
slack: the slack $n_*\T x - b_*$ grows at rate $n_*\T z > 0$ per unit step. Part
(i) says the step costs nothing in working-set feasibility. Together with the
closed forms~\eqref{eq:zclosed} and~\eqref{eq:rclosed} the reading is:
\begin{itemize}
  \item $z$ is the $G$-orthogonal projection of the unconstrained direction
        $G^{-1}n_*$ onto the working set's feasible subspace $\nullsp(A\T)$.
        Indeed $\Pi := J_2 J_2\T G$ is idempotent with range $\nullsp(A\T)$ and is
        self-adjoint for $\langle\cdot,\cdot\rangle_G$, and
        $z = \Pi\, G^{-1}n_*$.
  \item $r$ solves a least-squares problem in the $G^{-1}$ metric. By
        \eqref{eq:rclosed},
        \begin{equation}
          r = \argmin_{y \in \R^k} \; \norm{J\T A y - J\T n_*}_2
            = \argmin_{y \in \R^k} \; (Ay - n_*)\T G^{-1} (A y - n_*),
          \label{eq:lsq}
        \end{equation}
        the best approximation of the entering normal by the working-set normals.
        Its residual, mapped through $G^{-1}$, is $z$: from
        Lemma~\ref{lem:directions}(iii), $z = G^{-1}(n_* - Ar)$.
\end{itemize}

\begin{remark}[Why not call a least-squares routine]
Equation~\eqref{eq:lsq} is a genuine least-squares problem and it is tempting to
hand it to a library. The temptation should be resisted: the factor $R$ of its
normal equations is already held, so~\eqref{eq:dzr} costs one triangular solve,
where a library call would refactorise $J\T A$ from scratch at $O(nk^2)$ every
iteration. The same remark applies to~\eqref{eq:subsol}: the closed form is the
specification, not the implementation.
\end{remark}

\subsection{The affine path}

Fix the directions and consider moving along them. Let $u_*$ denote the entering
constraint's own multiplier, which starts at $0$.

\begin{proposition}[Dual feasibility along the path]\label{prop:path}
For $t \in \R$ define
\begin{equation}
  x(t) = x + t z, \qquad u(t) = u - t r, \qquad u_*(t) = u_* + t .
  \label{eq:path}
\end{equation}
Then for all $t$: $A\T x(t) = b_\Active$, and stationarity holds for the working
set augmented by the entering constraint,
\begin{equation}
  G x(t) - a = A\, u(t) + u_*(t)\, n_* .
  \label{eq:pathstat}
\end{equation}
\end{proposition}

\begin{proof}
The first claim is Lemma~\ref{lem:directions}(i). For the second, stationarity at
$t = 0$ reads $Gx - a = Au + u_* n_*$, and by Lemma~\ref{lem:directions}(iii),
\[
  Gx(t) - a = (Gx - a) + t\,Gz
  = A u + u_* n_* + t(n_* - Ar)
  = A(u - tr) + (u_* + t) n_* . \qedhere
\]
\end{proof}

So the entire path is stationary, with the entering multiplier rising at unit rate
and the working-set multipliers falling linearly along $-r$. Complementarity holds
for every constraint except the entering one, whose slack is still nonzero while
its multiplier is positive; that single violation is what the step is taken to
remove. Two limits on $t$ follow immediately.

\paragraph{The dual limit $t_1$.} Dual feasibility requires $u_i(t) \ge 0$ for
every inequality $i \in \Active$. Since $u_i(t) = u_i - t r_i$ and $u_i \ge 0$, the
binding indices are those with $r_i > 0$, and
\begin{equation}
  t_1 = \min\Bigl\{\, \frac{u_i}{r_i} \;:\; i \in \Active,\ i > m_{\mathrm{eq}},\
  r_i > 0 \,\Bigr\},
  \qquad
  i_{\mathrm{del}} = \text{an index attaining it},
  \label{eq:t1}
\end{equation}
with $t_1 = +\infty$ when the set is empty. Equality constraints are excluded:
their multipliers are unrestricted in sign, so no step in $t$ can make them
infeasible. This is the classical ratio test, and $t_1$ is the largest step that
preserves~\eqref{eq:invariant}.

\paragraph{The primal limit $t_2$.} The entering constraint's slack is
$n_*\T x(t) - b_* = s_{i_*} + t\, n_*\T z$ with $s_{i_*} < 0$, so it reaches zero
at
\begin{equation}
  t_2 = \frac{|s_{i_*}|}{n_*\T z} = \frac{|s_{i_*}|}{\norm{d_2}^2},
  \label{eq:t2}
\end{equation}
finite and positive whenever $z \neq 0$, and $t_2 = +\infty$ when $z = 0$ by
Lemma~\ref{lem:directions}(ii), in which case the primal iterate cannot move at all
and only the multipliers change.

\paragraph{The step.} Take $t = \min(t_1,t_2)$. If $t_2 \le t_1$ the step is
\emph{full}: the entering constraint now holds with equality, so it joins the
working set with multiplier $u_* + t_2$ and the factorisation is updated
(Section~\ref{sec:updates}). The new state satisfies~\eqref{eq:invariant}:
stationarity and complementarity by Proposition~\ref{prop:path} with the slack now
zero, and the sign conditions because $t_2 \le t_1$. If $t_1 < t_2$ the step is
\emph{partial}: the multiplier of $i_{\mathrm{del}}$ reaches zero first, that
constraint leaves the working set, the factorisation is downdated, and $z$ and $r$
are recomputed, both having changed with $J$ and $R$, keeping the same
entering constraint and its accumulated $u_*$, and the loop repeats. If both
limits are infinite the problem is infeasible; see
Proposition~\ref{prop:farkas}.

\begin{remark}[Equalities violated from above]\label{rem:reverse}
An equality constraint may be violated in either direction. When $s_{i_*} > 0$
the slack must be \emph{decreased}, so the step is taken along $-z$: formally
$t$ ranges over the negative reals, $t_2 = -|s_{i_*}|/n_*\T z$, the ratio test in
\eqref{eq:t1} runs against $-r$ rather than $r$, and the entering multiplier
accumulates negatively. Nothing else changes; in particular
Proposition~\ref{prop:path} was stated for all $t \in \R$ precisely so that this
case needs no separate treatment, and Proposition~\ref{prop:increment} below
covers it too.
\end{remark}

\subsection{The objective, in closed form}

The running objective is not re-evaluated from $x$; it is advanced by an exact
increment.

\begin{proposition}[Objective increment]\label{prop:increment}
Along the path~\eqref{eq:path},
\begin{equation}
  f(x(t)) - f(x) = t\Bigl(\frac{t}{2} + u_*\Bigr)\, n_*\T z ,
  \label{eq:increment}
\end{equation}
which is strictly positive whenever $z \neq 0$ and $t \neq 0$ and $t$, $u_*$ have
the same sign (in particular whenever $t > 0$ and $u_* \ge 0$).
\end{proposition}

\begin{proof}
Expanding the quadratic,
$f(x + tz) = f(x) + t\, z\T(Gx - a) + \tfrac12 t^2\, z\T G z$. For the linear
term, stationarity at $t = 0$ and Lemma~\ref{lem:directions}(i) give
$z\T(Gx - a) = z\T A u + u_*\, z\T n_* = u_*\, n_*\T z$. For the quadratic term,
Lemma~\ref{lem:directions}(ii) gives $z\T G z = n_*\T z$. Hence
$f(x+tz) - f(x) = t u_* n_*\T z + \tfrac12 t^2 n_*\T z$, which
is~\eqref{eq:increment}. Positivity: $n_*\T z > 0$ by
Lemma~\ref{lem:directions}(ii), and $t(t/2 + u_*) > 0$ when $t$ and $u_*$ share a
sign and $t \neq 0$.
\end{proof}

The increment~\eqref{eq:increment} is exact, so the objective can be carried as a
running total at the cost of one multiplication per step instead of an $O(n^2)$
re-evaluation of the quadratic. It is expressed in quantities the iteration
already holds, $n_*\T z$ being needed for $t_2$ anyway. And it handles the
reversed step of Remark~\ref{rem:reverse} without a
special case: there $t < 0$ and $u_*$ accumulates negatively, so
$t/2 + u_* < 0$ and the product with $t$ is again positive. The value therefore
increases on every nonzero step, in both directions, which is the monotonicity
Section~\ref{sec:termination} needs.

\begin{algorithm}[ht]
\caption{The Goldfarb--Idnani dual active-set method}\label{alg:gi}
\begin{algorithmic}[1]
\Require $G \succ 0$, $a$, $C$, $b$, $m_{\mathrm{eq}}$
\State $J \leftarrow \mathrm{chol}(G)^{-1}$; \ $x \leftarrow G^{-1}a$; \
       $f \leftarrow -\tfrac12 a\T x$; \ $\Active \leftarrow \emptyset$; \
       $k \leftarrow 0$; \ $u \leftarrow ()$
\Loop
  \State $s \leftarrow C\T x - b$, with $s_i \leftarrow 0$ for $i \in \Active$;
         choose $i_*$ by~\eqref{eq:select}
  \State \textbf{if} none \textbf{then} \Return $(x,f,u)$
         \Comment{optimal, by Prop.~\ref{prop:sufficient}}
  \State $n_* \leftarrow c_{i_*}$; \ $u_* \leftarrow 0$
  \Loop
    \State $d \leftarrow J\T n_*$; \ $z \leftarrow J_2 d_2$; \
           $r \leftarrow R^{-1} d_1$
    \State $t_1, i_{\mathrm{del}} \leftarrow$ \eqref{eq:t1}; \
           $t_2 \leftarrow$ \eqref{eq:t2};
           \textbf{if} both infinite \textbf{then stop}: infeasible
    \State $t \leftarrow \min(t_1,t_2)$; \ $x \leftarrow x + tz$; \
           $f \leftarrow f + t(t/2 + u_*)\,n_*\T z$; \
           $u \leftarrow u - tr$; \ $u_* \leftarrow u_* + t$
    \If{$t_2 \le t_1$}
      \State $\Active \leftarrow \Active \cup \{i_*\}$, $u_{k+1} \leftarrow u_*$,
             $k \leftarrow k+1$; insert (\S\ref{ssec:insert}); \textbf{break}
    \Else
      \State $\Active \leftarrow \Active \setminus \{i_{\mathrm{del}}\}$,
             $k \leftarrow k-1$; delete (\S\ref{ssec:delete});
             $s_{i_*} \leftarrow n_*\T x - b_*$
    \EndIf
  \EndLoop
\EndLoop
\end{algorithmic}
\end{algorithm}

%% file: sections/s5_termination.tex
\section{Termination}
\label{sec:termination}

Three questions are separate and are best kept so: does the inner loop stop, does
the outer loop stop, and what does it mean when the iteration gets stuck.

\subsection{The inner loop}

\begin{proposition}[The inner loop terminates]\label{prop:inner}
For a fixed entering constraint with $n_* \neq 0$, the inner loop performs at
most $k+1$ passes, where $k$ is the size of the working set on entry.
\end{proposition}

\begin{proof}
Every pass that does not exit performs a partial step and removes one constraint
from the working set, so after at most $k$ of them the working set is empty. With
$\Active = \emptyset$ we have $J_2 = J$ and hence, by
Lemma~\ref{lem:directions}(ii) and~\eqref{eq:inv1},
$n_*\T z = z\T G z$ with $z = J J\T n_* = G^{-1}n_*$, so
$n_*\T z = n_*\T G^{-1} n_* > 0$ because $G^{-1} \succ 0$ and $n_* \neq 0$. Thus
$z \neq 0$, $t_2$ is finite by~\eqref{eq:t2}, and $t_1 = +\infty$ because the
ratio test~\eqref{eq:t1} ranges over an empty set. The step is full and the loop
exits.
\end{proof}

The excluded case $n_* = 0$ is a column of $C$ that is identically zero. Such a
constraint reads $0 \ge b_*$: no $x$ influences it, so it is either vacuous
($b_* \le 0$) or renders the problem infeasible outright. The implementation
scores it as infinitely violated when $b_* > 0$, which routes it to the
infeasibility verdict below instead of to a division by zero
in~\eqref{eq:select}.

\subsection{The outer loop}

\begin{proposition}[Strict ascent and finiteness]\label{prop:ascent}
Every outer iteration ending in a full step with $t_2 \neq 0$ strictly increases
the objective, hence by Proposition~\ref{prop:dual} the dual value $\psi$. If every
full step has $t_2 > 0$, the method terminates after finitely many outer
iterations.
\end{proposition}

\begin{proof}
The iteration's total increment is the sum of~\eqref{eq:increment} over its inner
passes, each positive by Proposition~\ref{prop:increment}, since the accumulated $u_*$ and
each step $t$ share a sign throughout by Remark~\ref{rem:reverse}, and the
final term is nonzero because $t_2 \neq 0$. For the second claim, a working set
determines $(x,u)$ uniquely by Proposition~\ref{prop:sub}, hence determines the
objective; that value strictly increases from one outer iteration to the next, so
no working set recurs. There are finitely many subsets of $\{1,\dots,m\}$, and the
iteration stops only by the optimality test of~\eqref{eq:select} or by the
infeasibility verdict.
\end{proof}

\begin{remark}[Degeneracy]\label{rem:degeneracy}
The hypothesis fails exactly when a full step of length zero occurs, which requires
$s_{i_*} = 0$: more constraints pass through $x$ than the working set holds. A zero
step then adds one without changing $x$, $u$ or the objective, the ascent argument
goes silent, and a cycle is not excluded. Goldfarb and Idnani~\cite{goldfarb1983}
discuss this; the classical remedies are lexicographic or least-index rules. In
finite precision the question changes character, since a slack of exactly zero is
not what one observes.
\end{remark}

\subsection{Infeasibility is a certificate, not a failure}

The interesting case is a stuck iteration: $z = 0$, so the primal cannot move, and
$t_1 = \infty$, so no multiplier limits the step. The classical reading is that
the dual is unbounded, and therefore the primal infeasible. That argument needs
care: by Proposition~\ref{prop:path} the whole ray is dual feasible, and
by~\eqref{eq:increment} with $z = 0$ the value is constant. The cleanest
statement avoids the dual entirely: the vector $r$ that the solver already holds
is a Farkas certificate.

\begin{proposition}[Farkas certificate]\label{prop:farkas}
Suppose the iteration reaches a state in which the entering inequality
constraint $i_*$ is violated, $s_{i_*} < 0$, and
\begin{enumerate}
  \item[(a)] $z = 0$, and
  \item[(b)] $r_i \le 0$ for every inequality $i \in \Active$.
\end{enumerate}
Then the feasible set $\Omega$ is empty.
\end{proposition}

\begin{proof}
By (a) and Lemma~\ref{lem:directions}(iv), $n_* = A r = \sum_{i \in \Active} r_i
c_i$. Let $\tilde x \in \Omega$ be any feasible point. Splitting the working set
into its equality part $\mathcal{E}$ and inequality part $\mathcal{I}$,
\[
  n_*\T \tilde x
  = \sum_{i \in \mathcal{E}} r_i\, c_i\T \tilde x
    + \sum_{i \in \mathcal{I}} r_i\, c_i\T \tilde x
  \;\le\; \sum_{i \in \mathcal{E}} r_i b_i + \sum_{i \in \mathcal{I}} r_i b_i
  \;=\; r\T b_\Active ,
\]
where the equality members contribute $c_i\T\tilde x = b_i$ exactly, and each
inequality member contributes $r_i c_i\T \tilde x \le r_i b_i$ because
$c_i\T\tilde x \ge b_i$ and $r_i \le 0$ by (b). On the other hand the current
iterate $x$ satisfies $A\T x = b_\Active$, so
\[
  r\T b_\Active = r\T A\T x = n_*\T x = s_{i_*} + b_* < b_* .
\]
Combining, $n_*\T\tilde x < b_*$, contradicting $\tilde x \in \Omega$. Hence
$\Omega = \emptyset$.
\end{proof}

The proposition says the solver's infeasibility verdict is a proof rather than an
exhausted search, and it says what the proof is: the multipliers $r$ of the
entering normal in terms of the working-set normals, which the solver computed for
its own purposes, are the Farkas multipliers. The reversed case of
Remark~\ref{rem:reverse}, an equality violated from above, follows by
applying the proposition to the constraint $-n_*\T x \ge -b_*$, which the equality
implies.

Hypothesis (a) is $t_2 = \infty$ and (b) is $t_1 = \infty$, so the condition an
implementation tests, neither step limit being finite, is literally the hypothesis
of Proposition~\ref{prop:farkas}.

\begin{remark}[The remaining hypothesis is not free in floating point]
\label{rem:spurious}
Proposition~\ref{prop:farkas} also needs the entering constraint to be genuinely
violated. Exact arithmetic gives that: a constraint is selected only when
$v_{i_*} > 0$. Floating point does not, and the gap is reachable by the very
substitution Section~\ref{ssec:insert} argues for. A Householder reflection and
a Givens chain leave the iterate at slightly different points, so a constraint one
implementation leaves inside its boundary by a few multiples of the machine
epsilon, another can leave the same distance outside. Reaching the stuck state with
such a constraint selected, a solver must neither enforce it, there being nothing
to enforce, nor conclude infeasibility, the conclusion not following; the
implementation studied here sets it aside for the current iterate and takes the
next candidate, discarding that judgement as soon as the iterate moves. The test
separating rounding from a real violation is comparatively easy to set, because it
need only separate rounding from \emph{provable} infeasibility, and infeasibility
is macroscopic: its size is set by the geometry of the constraints rather than by
the arithmetic, so the two populations lie orders of magnitude apart.
\end{remark}

%% file: sections/s6_updates.tex
\section{Updating the factorisation}
\label{sec:updates}

Each iteration changes the working set by one column, and the factors must follow.
Recomputing them from scratch costs a QR factorisation of $J\T A$, $O(nk^2)$; the
updates below cost $O(n^2)$ and $O(nk)$ respectively. Over the $O(n)$ iterations
of a run that is the difference between $O(n^3)$ and $O(n^4)$, and the whole reason
the factors are carried and not rebuilt.

Both updates work the same way. The invariant~\eqref{eq:inv1} is preserved by
post-multiplying $J$ with any orthogonal $W$, since $(JW)\T G (JW) = W\T W = I$;
the update's task is to choose $W$ so that~\eqref{eq:inv2} is restored for the new
working set. Because $J' = JW$ gives ${J'}\T A = W\T(J\T A)$, the same $W\T$ acts on
the rows of $R$.

\subsection{Insertion}
\label{ssec:insert}

Let the working set of size $k$ be extended by the entering constraint, so
$A' = [\,A \ \ n_*\,]$. From~\eqref{eq:inv2} and $d = J\T n_*$,
\begin{equation}
  J\T A' = \begin{bmatrix} R & d_1 \\ 0 & d_2 \end{bmatrix},
  \label{eq:insert1}
\end{equation}
which fails to be triangular only in the $n-k$ entries of $d_2$ below its first.
Choose an orthogonal $W_2 \in \R^{(n-k)\times(n-k)}$ with
$W_2\T d_2 = \alpha e_1$ and set $W = \mathrm{diag}(I_k, W_2)$. Then
\begin{equation}
  (JW)\T A' = \begin{bmatrix} R & d_1 \\ 0 & W_2\T d_2\end{bmatrix}
  = \begin{bmatrix} R' \\ 0 \end{bmatrix},
  \qquad
  R' = \begin{bmatrix} R & d_1 \\ 0 & \alpha \end{bmatrix},
  \label{eq:insert2}
\end{equation}
upper triangular of order $k+1$, and the leading $k$ columns of $J$ are untouched
because $W$ acts as the identity on them.

\begin{proposition}[Full column rank is maintained]\label{prop:rank}
The insertion is performed only after a full step, which requires $t_2 < \infty$
and hence $z \neq 0$. Then $\alpha \neq 0$, so $R'$ is nonsingular and $A'$ has
full column rank $k+1$.
\end{proposition}

\begin{proof}
$|\alpha| = \norm{W_2\T d_2} = \norm{d_2}$, which is nonzero iff $z \neq 0$ by
Lemma~\ref{lem:directions}(ii). $R'$ is triangular with diagonal that of $R$
extended by $\alpha$, so it is nonsingular, and ${R'}\T R' = {A'}\T G^{-1}A'$ is then
positive definite, which forces $A'$ to have full column rank.
\end{proof}

The algorithm therefore never has to test for linear dependence among the
working-set normals. The step rules exclude it. A
constraint whose normal lies in the span of those already held has $z = 0$ by
Lemma~\ref{lem:directions}(iv), so its step is limited by $t_1$ and it can only
be added after enough partial steps have removed the dependence. This is the
structural advantage that Section~\ref{sec:pdas} gives up.

\paragraph{The choice of $W_2$.} Any orthogonal $W_2$ with
$W_2\T d_2 \in \R e_1$ will do, and by Proposition~\ref{prop:invariance} the
choice does not affect the iterates. The reference implementation uses a chain of
$n-k-1$ Givens rotations, one per trailing component. A single Householder
reflection achieves the same reduction:
\begin{equation}
  W_2 = I - \frac{2}{\beta}\, v v\T,
  \qquad v = d_2 - \alpha e_1,
  \qquad \beta = v\T v,
  \qquad \alpha = \pm\norm{d_2}.
  \label{eq:householder}
\end{equation}
That $W_2 d_2 = \alpha e_1$ is the standard Householder identity, using
$\beta = 2\,v\T d_2$. The reflection is symmetric, so it applies to the columns of
$J$ and to $d_2$ as one operation, and applying it to a block is a
matrix--vector product followed by a rank-one update, two operations of
$O(n(n-k))$, where the Givens chain is $n-k-1$ operations of $O(n)$. The
arithmetic is comparable; the number of operations is not.

\begin{remark}[The sign of $\alpha$, and cancellation]\label{rem:sign}
Numerical analysis prescribes $\alpha = -\mathrm{sign}((d_2)_1)\norm{d_2}$, which
makes $v_1 = (d_2)_1 - \alpha$ a sum of like-signed terms. The implementation
takes the opposite sign, $\alpha = \mathrm{sign}((d_2)_1)\norm{d_2}$, so as to
reproduce the sign convention of the reference's Givens chain, and thereby walks
into the cancellation the prescription exists to avoid, severely when $d_2$ is
already close to a positive multiple of $e_1$. The cure is algebraic rather than a
change of convention: with $\tau = \norm{(d_2)_{2:}}^2$ and
$\nu = \norm{d_2} = \sqrt{(d_2)_1^2 + \tau}$,
\begin{equation}
  v_1 = (d_2)_1 - \mathrm{sign}((d_2)_1)\,\nu
      = \frac{(d_2)_1^2 - \nu^2}{(d_2)_1 + \mathrm{sign}((d_2)_1)\nu}
      = \frac{-\mathrm{sign}((d_2)_1)\,\tau}{|(d_2)_1| + \nu},
  \label{eq:nocancel}
\end{equation}
in which every operation combines like-signed quantities. The remaining entries
of $v$ are those of $d_2$, and $\beta = \tau + v_1^2$. By
Proposition~\ref{prop:invariance} the sign convention is anyway immaterial to the
iterates; \eqref{eq:nocancel} is what makes it immaterial to the accuracy as
well. See~\cite{golub2013,higham2002} for the standard treatment.
\end{remark}

\begin{remark}[Why the reflections cannot be blocked]\label{rem:wy}
A sequence of Householder reflections is normally applied as a block. The WY
representation of Bischof and Van Loan~\cite{bischof1987}, and its storage-efficient
form due to Schreiber and Van Loan~\cite{schreiber1989}, write a product
$W_1 W_2 \cdots W_j$ as $I + WY\T$, so that $j$ reflections reach a matrix in one
level-3 update rather than $j$ level-2 ones. Blocked QR factorisation is fast for
that reason, and the first thing to reach for here.

It is not available, for a structural reason rather than an implementation one.
Blocking requires the $j$ reflections to be known before any is applied. In this
iteration the reflection that inserts a constraint is determined by
$d = J\T n_*$, which depends on $J$ \emph{after} the previous insertion; and
between any two insertions the solver reads $z$ and $r$ off the updated factors to
decide which constraint enters next, and whether it enters at all. The updates are
interleaved with reads of what they produce, so there is no window in which two
reflections are simultaneously known and unapplied.

This is the same obstruction as in the deletion chase below, and the reason the
method stays a level-2 computation no matter how it is coded. A variant
that inserted several constraints per iteration would open the window; the dual
method's step rules, which admit one constraint at a time by construction, close
it.
\end{remark}

\subsection{Deletion}
\label{ssec:delete}

Let the constraint in position $p$ of the working set be removed. Deleting column
$p$ of $A$ deletes column $p$ of $R$ and shifts columns $p+1,\dots,k$ one place
left, each retaining its own entries. The result is upper triangular except for
one subdiagonal: writing $\tilde R$ for the shifted $k \times (k-1)$ matrix,
$\tilde R_{ij} = 0$ for $i > j+1$, with the offending entries
$\tilde R_{j+1,j}$ for $j = p,\dots,k-1$.

These entries lie in distinct columns, and no single reflection annihilates
components of distinct vectors. What restores the shape is a \emph{chase} of
$k-p$ plane transformations: for $j = p,\dots,k-1$ in order, a $2\times2$
orthogonal acting on rows $j$ and $j+1$ annihilates $\tilde R_{j+1,j}$, and the
same transformation is applied to columns $j$ and $j+1$ of $J$. The chase is
inherently sequential, the parameters of each transformation being read from entries
the previous one wrote, which is why no batched form of it exists.

The implementation takes the $2 \times 2$ block to be the symmetric reflection
\begin{equation}
  W_j^{(2)} = \begin{bmatrix} c & s \\ s & -c \end{bmatrix},
  \qquad c = \frac{x}{h},\quad s = \frac{y}{h},
  \qquad h = \mathrm{sign}(x)\sqrt{x^2+y^2},
  \label{eq:refl2}
\end{equation}
where $x = \tilde R_{jj}$ and $y = \tilde R_{j+1,j}$, so that
$W_j^{(2)}(x,y)\T = (h,0)\T$. Symmetry is the point. The update rule of this
section applies $W$ to the columns of $J$ and $W\T$ to the rows of $R$; a
symmetric block makes those the same matrix, so one $2\times2$ serves both sides.
A Givens rotation, which is what BLAS offers, is not symmetric: it agrees
with~\eqref{eq:refl2} on the first output and negates the second, so it requires
its transpose on the $R$ side. By Proposition~\ref{prop:invariance} either choice
gives the same iterates, the difference being a sign matrix.

\begin{remark}[The degenerate step]
When $x = 0$ the formula~\eqref{eq:refl2} gives $c = 0$, $s = \pm1$, and the
transformation reduces to an exchange of the two rows up to sign. The
implementation performs a plain exchange, $\left[\begin{smallmatrix} 0 & 1 \\ 1 &
0\end{smallmatrix}\right]$, which is orthogonal and symmetric and annihilates $y$
just as well; it differs from~\eqref{eq:refl2} by a sign when $y < 0$, and
Proposition~\ref{prop:invariance} says that is free.
\end{remark}

\begin{remark}[Where the storage cost lands]
$R$ is held as packed columns, entry $(i,j)$ with $i \le j$ at offset
$j(j+1)/2 + i$, so its leading $k \times k$ block, which the solve for $r$ reads
once per iteration, is a contiguous run at every $k$. Insertion writes one column,
contiguous in that layout; deletion mixes two \emph{rows} across a range of
columns, which is a strided gather because column offsets grow with the index. The
asymmetry is deliberate, and the right way round. It is a property of the
layout rather than of the derivation, to which the mathematics is indifferent.
\end{remark}

%% file: sections/s8_pdas.tex
\section{Guessing the active set instead of walking to it}
\label{sec:pdas}

The dual method reaches the active set one constraint at a time. That is its
defining feature and, at moderate $n$, its principal cost: a budget-plus-bounds
problem in $n = \qpPdasRefN{}$ variables takes some \qpPdasBudgetOuter{} outer
iterations (Section~\ref{ssec:pdas}), each of which is a handful of
matrix--vector products. The alternative is to guess the whole active
set at once, solve the working-set subproblem it induces, and repair the guess
from the signs that come back. This is the semismooth Newton method of
Hinterm\"uller, Ito and Kunisch~\cite{hintermuller2002} read as a primal--dual
active-set strategy, and block principal pivoting on the associated linear
complementarity problem~\cite{judice1994,murty1988,cottle1992} read as a rule for
exchanging several indices at a time.

\subsection{The iteration}

Given a candidate set $\Active$, Proposition~\ref{prop:sub} gives the working-set
solution directly:
\begin{equation}
  \bigl(C_\Active\T G^{-1} C_\Active\bigr)\lambda_\Active
   = b_\Active - C_\Active\T x_u,
  \qquad
  x = x_u + G^{-1}C_\Active \lambda_\Active,
  \qquad x_u = G^{-1}a,
  \label{eq:pdassolve}
\end{equation}
with $\lambda_i = 0$ for $i \notin \Active$. Both solves go through one Cholesky
factorisation of $G$, computed once, and one of the $k \times k$ dual Hessian,
computed per repair. The repair rule reads the two sign conditions of
\eqref{eq:dfeas}--\eqref{eq:pfeas} as instructions:
\begin{equation}
  \Active' = \{1,\dots,m_{\mathrm{eq}}\} \;\cup\;
  \bigl\{\, i > m_{\mathrm{eq}} \;:\;
    (i \in \Active \text{ and } \lambda_i \ge 0)
    \ \text{ or }\
    (c_i\T x < b_i) \,\bigr\},
  \label{eq:repair}
\end{equation}
that is: an active constraint whose multiplier has gone negative does not belong
in the set and leaves it; an inactive constraint that is violated at $x$ does
belong and joins it; equalities are always in. The starting guess is the
equalities together with whatever the unconstrained minimiser violates, which is
already the answer when it violates nothing. If $\Active' = \Active$ the
conditions~\eqref{eq:pfeas}--\eqref{eq:comp} hold to the tolerance used, and
\eqref{eq:pdassolve} supplies \eqref{eq:stat} by construction.

The whole set is exchanged at once, and it converges in a handful of
repairs almost independently of $n$, from \qpPdasRepairsLo{} to
\qpPdasRepairsHi{} across every family and size measured in
Section~\ref{ssec:pdas}. It trades arithmetic, which is abundant at these sizes,
for iterations, which are not.

\subsection{Why there is no finite-termination theorem here}

Block exchange can cycle, and the standard remedy is a least-index rule that flips
only the lowest-indexed offending constraint, in the manner of Bland's rule for
the simplex method and Murty's least-index rule for linear complementarity
problems~\cite{murty1988}. What follows concerns this \emph{block} exchange
specifically, and not active-set methods in general: the exact walk of
Section~\ref{sec:iteration} is finite under nondegeneracy
(Proposition~\ref{prop:ascent}), and for parametric families of QPs its worst-case
iteration count can be certified exactly~\cite{arnstrom2022certification}. In the bound-constrained case that remedy comes with a
theorem. The theorem does not survive the generalisation, and the reason is
structural, not a gap in the argument.

Consider $\min_{x \ge 0} \tfrac12 x\T G x - a\T x$, so that $C = I$ and $m = n$.
The system solved on a candidate set $\Active$ is then
\begin{equation}
  \bigl(I_{\cdot\Active}\T G^{-1} I_{\cdot\Active}\bigr)\lambda_\Active
  = \bigl(G^{-1}\bigr)_{\Active\Active}\lambda_\Active
  = b_\Active - x_{u,\Active},
  \label{eq:principal}
\end{equation}
whose coefficient matrix is a \emph{principal submatrix} of $G^{-1}$. Since
$G^{-1} \succ 0$, every principal submatrix of it is positive definite, so
$G^{-1}$ is a $P$-matrix: every principal pivot is well defined, whatever the
guess. That is the hypothesis under which block principal pivoting with a
least-index guard terminates finitely at the unique
solution~\cite{judice1994,murty1988}, and the guarantee survives inexact inner
solves once their residual is small against the problem's decision
margin~\cite{schmelzer2026nncg}.

None of it is available for general $C$. The matrix
$C_\Active\T G^{-1} C_\Active$ is positive definite exactly when $C_\Active$ has
full column rank, and that is a property of the guess, not of the data: a guessed
set may name $k > n$ constraints, or $k \le n$ linearly dependent ones, and then
the pivot is undefined. There is no $P$-matrix to
appeal to, and correspondingly no theorem to inherit. Contrast
Proposition~\ref{prop:rank}: the dual method's step rules make dependence
unreachable, and guessing forfeits that protection.

That the plain exchange fails on general constraints is not news, and the repair
the literature offers says where the difficulty sits. Curtis, Han and
Robinson~\cite{curtis2015pdas} recover global convergence by carrying an index set
auxiliary to the active-set estimate, which holds the indices that would otherwise
cycle. That changes what the method carries, not the order in which it exchanges,
and it has to: no rule for choosing the next set can recover the bound-constrained
theorem when the subproblem on the chosen set may be singular. What is claimed here
is the mechanism and its price. The rank condition is the one already governing
Proposition~\ref{prop:sub}, it is a property of the guess and not of the data, and
Section~\ref{ssec:pdas} measures both how invisible it is on random instances and
how quickly a repeated column makes it certain.

The implementation treats the Cholesky factorisation of
$C_\Active\T G^{-1}C_\Active$ as the rank test, where a dependent guess fails
instead of returning something plausible, and the least-index rule as a way of
making progress where block exchange oscillates, and not as a termination proof.
Neither substitutes for the guarantee. The certificate does.

\subsection{The certificate}

Because the method may fail, and may fail by stabilising on a set that is simply not
optimal, the trade is sound only because the answer is checked. Here
Proposition~\ref{prop:sufficient} carries the argument: a candidate
$(x,\lambda)$ satisfying
\begin{equation}
  \norm{G x - a - C\lambda}_\infty \le \varepsilon,
  \quad |c_i\T x - b_i| \le \varepsilon \ (i \le m_{\mathrm{eq}}),
  \quad c_i\T x - b_i \ge -\varepsilon, \ \ \lambda_i \ge -\varepsilon,
  \ \ |\lambda_i (c_i\T x - b_i)| \le \varepsilon
  \label{eq:certificate}
\end{equation}
on the inequalities, with $\varepsilon$ scaled by the data, \emph{is} the answer to
that tolerance.
Stationarity is checked against $G$ directly and not trusted from the
construction, since the construction is what an ill-conditioned working set
corrupts. A candidate that fails is discarded and the exact walk of
Algorithm~\ref{alg:gi} runs instead, so guessing can never degrade the answer: it
produces the minimiser or nothing.

\begin{remark}[Two tolerances with very different jobs]
The tolerance in~\eqref{eq:repair} decides which set is tried next. A bad choice
there costs a repair or a fallback and can never cost correctness, so it is not
delicate. The tolerance in~\eqref{eq:certificate} is the only thing standing
between a non-optimal point and the caller, and is therefore the one to reason
about. It is nonetheless not delicate either, because the quantity it thresholds
is bimodal: a candidate constructed from a well-conditioned working set satisfies
\eqref{eq:certificate} to a few multiples of the machine epsilon, and one built on
a set that is not has no reason to come close. Section~\ref{ssec:pdas} measures
the accepted population, whose worst residual is $\qpPdasAcceptWorst$, and reports
what it did and did not observe on the rejecting side.
\end{remark}

%% file: sections/s9_sweep.tex
\section{A family of programs differing only in the linear term}
\label{sec:sweep}

An efficient frontier, a rolling rebalance and a scenario grid all solve the same
program repeatedly with a slightly different linear term: $G$, $C$, $b$ and
$m_{\mathrm{eq}}$ are fixed and only $a$ varies. Solved independently, each
instance rediscovers an active set it almost always already had.

What makes this exploitable is a reading of the invariants
\eqref{eq:inv1}--\eqref{eq:inv2}: \emph{neither factor depends on $a$}. $J$
depends on $G$ alone through $J\T G J = I$ and on the working set through the
orthogonal transformations accumulated into it; $R$ depends on $G$ and the working
set. The linear term enters only through $x_u = G^{-1}a$ and the right-hand sides.
So across such a family both factors are reusable verbatim, and the entire
solution can be recovered from them.

\subsection{Recovery}

\begin{proposition}[Recovery from cached factors]\label{prop:recover}
Let $(J,R)$ satisfy \eqref{eq:inv1}--\eqref{eq:inv2} for a working set $\Active$
of size $k$ with normals $A$. For a new linear term $a$, set
\begin{equation}
  x_u = J\,(J\T a),
  \qquad
  \rho = b_\Active - A\T x_u,
  \qquad
  R\T y = \rho,
  \qquad
  x = x_u + J_1 y,
  \qquad
  R\,u = y .
  \label{eq:recover}
\end{equation}
Then $(x,u)$ is the solution and multiplier pair of the working-set subproblem
\eqref{eq:sub}.
\end{proposition}

\begin{proof}
$x_u = JJ\T a = G^{-1}a$ by~\eqref{eq:inv1}. By Proposition~\ref{prop:blocks}(i),
$u = R^{-1}y = R^{-1}R^{-\top}\rho = (R\T R)^{-1}\rho = (A\T G^{-1}A)^{-1}\rho$,
which is the multiplier of~\eqref{eq:subsol}. For the iterate, $Ru = y$ gives
$J_1 y = J_1 R u = G^{-1}A u$ by Proposition~\ref{prop:blocks}(iii), so
$x = x_u + G^{-1}Au$, again as in~\eqref{eq:subsol}.
\end{proof}

The cost is $O(n^2)$, all of it in $x_u = J(J\T a)$: two matrix--vector products with
a dense $n \times n$ matrix, $J$ having lost its triangularity at the first update.
Everything after that is $O(nk + k^2)$. Re-deriving the factors for the
same working set instead costs $k$ insertions at $O(n^2)$ each, so the saving is a
factor of order $k$, largest exactly where it matters: a long-only optimum is a
vertex at which most bounds bind.

The exponent is easy to get wrong, because the natural experiment cannot see it.
On a box family the active set grows with $n$, so $O(nk)$ and $O(n^2)$ predict the
same curve and either reading fits; holding $k$ fixed at $\qpSweepFixedK{}$ while
$n$ grows separates them, and the measured local exponent
$d\log t/d\log n$ then runs from $\qpSweepSlopeLo$ at $n = \qpSweepNlo{}$ to
$\qpSweepSlopeHi$ at $n = \qpSweepNhi{}$. Neither is $1$. The flat small-$n$ end is
an overhead floor, a hit there costing what its dozen array operations cost to
dispatch, and it ends where the $n^2$ overtakes dispatch.

Whether $(x,u)$ from~\eqref{eq:recover} is the answer to the \emph{full} program
is then one certificate away, the same certificate as before: the
multipliers of the working-set inequalities must be non-negative and no constraint
outside the working set may be violated. By
Proposition~\ref{prop:sufficient} that is a proof. Note the asymmetry in cost: the
recovery is $O(n^2)$ but the check must look at every constraint, so it is
$O(nm)$ for a dense $C$, so verification and not recovery is what dominates a
successful reuse on a dense problem.

\subsection{Repair rather than restart}

When the check fails, the cached working set is stale, but it remains a far better
starting point than the unconstrained minimum. What blocks resuming from it is
\eqref{eq:invariant}: the iteration requires dual feasibility, and a stale set
typically has one or more negative multipliers. Those are precisely the
constraints that no longer belong.

\begin{proposition}[Repair terminates in a resumable state]\label{prop:repair}
Iterate the following from the cached set: recover $(x,u)$ by
\eqref{eq:recover}; if every inequality multiplier is non-negative, stop;
otherwise drop one constraint with a negative multiplier, downdate $(J,R)$ by
Section~\ref{ssec:delete}, and repeat. The loop stops after at most $k$ passes, in
a state satisfying~\eqref{eq:invariant}.
\end{proposition}

\begin{proof}
Each pass that does not stop shrinks the working set by one, so there are at most
$k$; with an empty set the sign condition is vacuous, which is the cold start of
Remark~\ref{rem:cold}. On exit $(x,u)$ solves the subproblem by
Proposition~\ref{prop:recover} and the sign conditions hold,
which is~\eqref{eq:invariant}; full column rank is inherited, deleting columns
preserving it.
\end{proof}

From such a state the iteration of Section~\ref{sec:iteration} cannot tell a
resumed solve from a cold one, the invariant being all it reads, so the resumed
walk needs only to drive the remaining primal infeasibility to zero. In the worst
case everything is dropped and the resumed walk is a cold solve, so the repair is
never worse than restarting up to the cost of the drops themselves.

\begin{remark}[Why this cannot change the answer]
Neither the reuse nor the repair introduces an approximation. The reuse returns a
point only when it carries a certificate, and the repair merely chooses a
different starting state for an iteration whose output is determined by
Proposition~\ref{prop:sufficient}. What does change is the reported iteration
count, which counts the working-set edits actually performed and is therefore not
comparable with a cold solve's.
\end{remark}

%% file: sections/s11_experiments.tex
\section{Numerical experiments}
\label{sec:experiments}

Everything above is a proof, and proofs leave two things open. The first is whether
the code satisfies the identities: they hold in exact arithmetic, but the
implementation reaches them through a triangular factor held as packed columns
addressed by hand and through library calls writing in place into a view of $J$'s
own buffer, correct only while that view stays contiguous. A sign error in the
packed indexing would leave every proposition above true and the running solver
wrong, and wrong quietly: the invariants are maintained across a run, not
recomputed, and a solver drifting from $J\T G J = I$ still returns a vector that is
very nearly right. The second is the set of quantities the note reasons about but
cannot predict: how often a guessed active set is certifiable, how many repairs it
takes, what a warm-started solve costs. All measurements are on the implementation
at the version recorded in Section~\ref{sec:reproducibility}, by scripts
distributed with this note's sources.

\subsection{The guessed active set, and what the certificate catches}
\label{ssec:pdas}

Section~\ref{sec:pdas} makes two claims only measurement settles: that the block
exchange converges in a number of repairs independent of $n$, and that losing the
$P$-matrix property for general $C$ is a real exposure.
Figure~\ref{fig:pdas} addresses both over \qpPdasFamilies{} constraint families,
\qpPdasInstances{} instances each, at $n \in \{\qpPdasSizes\}$, instrumenting the
shipped routines instead of reimplementing them.

\begin{figure}[ht]
\centering
\includegraphics[width=0.86\textwidth]{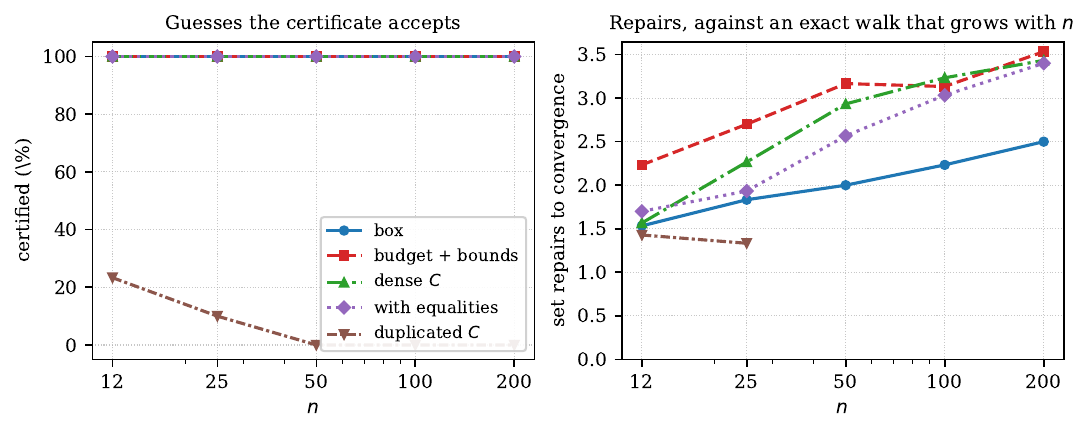}
\caption{Left: the fraction of guesses the certificate accepts. Right: set repairs
to convergence. Four families are certified without exception at every size and
settle in \qpPdasRepairsLo{} to \qpPdasRepairsHi{} repairs while the exact walk's
outer iterations grow linearly in $n$. The fifth is constructed so that a guessed
set can be linearly dependent, and it fails.}
\label{fig:pdas}
\end{figure}

On the four well-posed families (bounds, budget-plus-bounds, dense $C$, and $C$
with equalities) every guess is certified at every size, and the repair
count rises only from \qpPdasRepairsLo{} to \qpPdasRepairsHi{} across a
sixteen-fold range of $n$, never exceeding \qpPdasRepairsMax{}. The comparison
that matters is against the exact walk on the same instances: on
budget-plus-bounds it takes \qpPdasBudgetOuter{} outer iterations at
$n = \qpPdasRefN{}$, and its count grows linearly where the repair count is very
nearly flat. Section~\ref{sec:pdas} describes that trade; here it is measured.

\paragraph{The rank-deficient guess is real, and invisible by accident.} Over
\qpPdasInstances{} instances at each of five sizes, not one guess on those four
families produced a rank-deficient working-set system, from which a reader might
conclude the exposure is theoretical. It is not, merely invisible in random data,
where a dense $C$ almost never has a dependent subset the guess happens to name.
The fifth family makes it reachable in the most elementary way, by repeating
columns of $C$, so that naming both copies induces a singular system: there the
rank test fires on \qpPdasSingularDup\% of attempts and the certified fraction
collapses from \qpPdasCertSmall\% at $n = \qpPdasNsmall{}$ to \qpPdasCertBig\% at
$n = \qpPdasNbig{}$, against \qpPdasSingularBox\% on the others. Duplicated
constraints are not exotic, arising whenever a model is assembled programmatically
from overlapping rules, and they give the structural point of
Section~\ref{sec:pdas} its concrete form: for bounds the system is a principal
submatrix of $G^{-1}$ and cannot be singular, and nothing about general $C$
preserves that.

\paragraph{Where the failures are caught.} Of the candidates reaching the
certificate, \qpPdasAccepted{} were accepted with a worst KKT
residual~\eqref{eq:certificate} of $\qpPdasAcceptWorst$, and \qpPdasRejected{}
were rejected. We report that asymmetry instead of smoothing it: on these
families the certificate never had to reject a converged candidate, because the
failures were caught earlier and more cheaply by the Cholesky factorisation
refusing a dependent guess. The certificate remains the guarantee, but on this
evidence the rank test does the work in practice.

\subsection{Against other solvers, and against the same method done differently}
\label{ssec:compare}

The experiments so far compare the solver with its own claims, which establishes
that it is implemented correctly and says nothing about whether the method is worth
using. Table~\ref{tab:compare} and Figure~\ref{fig:compare} answer that against
four alternatives. Two are outside the family: an operator-splitting method (OSQP)
and an interior-point method (Clarabel). Two are inside it, and are the
informative ones: the reference C implementation of this exact algorithm, and
DAQP~\cite{arnstrom2022daqp}, which walks the same dual active set but carries it
on recursive $LDL^\top$ updates of the working-set system rather than on the
$(J,R)$ pair of Section~\ref{sec:factorisation}.

Times alone would mislead, in the direction that flatters whichever solver was
configured loosest: an active-set method returns an exact KKT point, a splitting
method stops at a tolerance, and an interior-point method approaches the boundary
without reaching it. Every solver is therefore asked for $\qpCmpTol$ and the
residual it \emph{achieves} is reported beside its time, in the scaled sup-norm
of~\eqref{eq:certificate}. Multipliers are recovered from the returned iterate by
least squares and not read from the solver, so none is charged for reporting
them in another convention.

\begin{table}[ht]
\centering
\small
\caption{Time per solve in milliseconds and achieved KKT residual, at
$n \in \{\qpCmpSizes\}$. The first four rows are active-set methods and the last
two are not, which the residual column is there to show.}
\label{tab:compare}
\begin{tabular}{lrrrrrrrr}
\toprule
& \multicolumn{2}{c}{$n = 50$} & \multicolumn{2}{c}{$n = 100$}
& \multicolumn{2}{c}{$n = 200$} & \multicolumn{2}{c}{$n = 400$} \\
\cmidrule(lr){2-3}\cmidrule(lr){4-5}\cmidrule(lr){6-7}\cmidrule(lr){8-9}
Solver & ms & resid. & ms & resid. & ms & resid. & ms & resid. \\
\midrule
\quadprogCompareRows
\bottomrule
\end{tabular}
\end{table}

\begin{figure}[ht]
\centering
\includegraphics[width=0.86\textwidth]{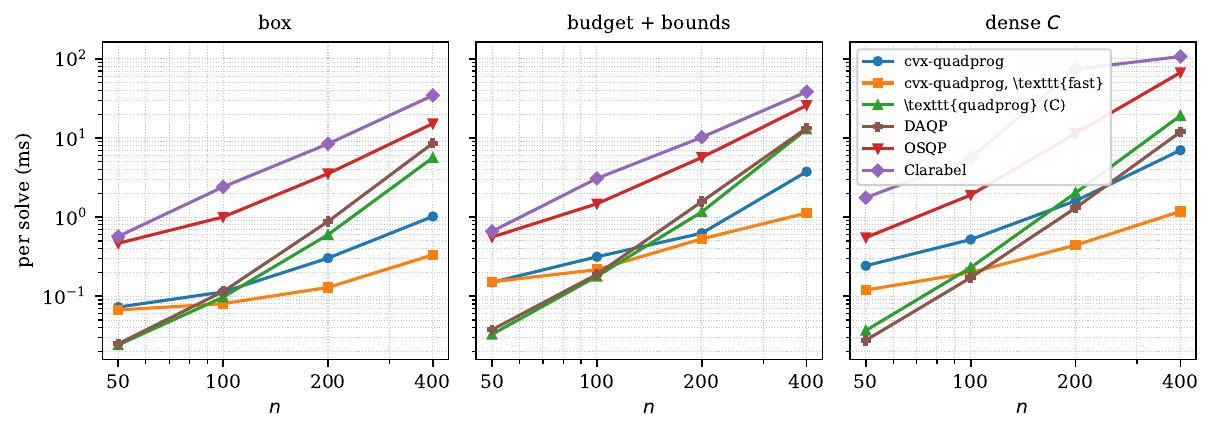}
\caption{Time per solve against $n$ on three constraint shapes. The reference
implementation of the same algorithm wins at the small end, where cost is
interpreter dispatch and not arithmetic, and loses at the large end, where it
is arithmetic.}
\label{fig:compare}
\end{figure}

\paragraph{The same algorithm in two languages crosses over.} The C reference is
the fastest solver here at $n = 50$ and among the slowest at
$n = \qpCmpNhi{}$, where it is $\qpCmpVsRef\times$ slower than the
implementation measured here. Nothing algorithmic separates them: they
walk the same active sets in the same order. What separates them is that below a
few hundred variables a solve costs what its operations cost to dispatch, and
above it a solve costs arithmetic, which reaches vendor-tuned kernels in one and
hand-written scalar loops in the other. This reproduces, independently and as a
by-product, the crossover the implementation is designed around.

\paragraph{Accuracy is where the interior-point method pays.} Clarabel is asked
for $\qpCmpTol$ and delivers between $10^{-9}$ and
$\qpCmpResidClarabel$, against $\qpCmpResidExact$ for the active-set method,
several orders of magnitude, and not a defect of the solver. The optima here are
vertices of the feasible polyhedron, with many constraints active, and a vertex is
exactly the case an interior-point method approaches asymptotically. OSQP reaches
$\qpCmpResidOsqp$, comparable to the active-set method, but only with its
polishing step enabled and its tolerance set to $\qpCmpTol$; at its default
tolerance it is faster and not comparable. The general point is that on
combinatorially structured optima an active-set method returns an answer the
others can only approach, and comparing times without residuals hides that.

\paragraph{A different factorisation of the same walk crosses over too.} Both
in-family solvers behave the same way. Both are fastest of everything at $n = 50$,
the C reference on bounds and on budget-plus-bounds and DAQP on dense $C$, and both
are behind at
$n = \qpCmpNhi{}$, where DAQP is $\qpCmpVsDaqp\times$ slower on bounds and
$\qpCmpVsDaqpDense\times$ on dense $C$. Accuracy is indistinguishable from ours
throughout ($\qpCmpResidDaqp$ at worst), as one expects of methods in the same
family: all three terminate on a vertex rather than approaching one, which is
exactly what separates them from the two rows below. The comparison is not
apples to apples in one respect: DAQP is written in library-free C
for embedded targets and pays no interpreter cost, so at small $n$ it is measuring
a different thing from us. What it shares is the shape of the crossover, and what
it shows is that the shape is a property of the regime and not of any one
implementation.

\paragraph{Guessing the active set is worth about what the language is.} The
\texttt{fast} path is the quickest of everything measured at
$n = \qpCmpNhi{}$ on every family, and it beats the exact walk by
$\qpCmpFastLo\times$ to $\qpCmpFastHi\times$. That is the same order as the gap
between the two implementations of the exact walk at the same size,
$\qpCmpVsRef\times$ on bounds and $\qpCmpVsRefDense\times$ on dense $C$, and the
two orderings run opposite: guessing gains most on dense $C$, where the language
gains least, and least on bounds, where it gains most. Iterations, not the cost of
an iteration, dominate at these sizes; Section~\ref{sec:pdas}'s argument arrives
here from a different direction.

\paragraph{What this comparison is not.} These are dense problems of small to
moderate size whose optima are vertices, which is the territory
Section~\ref{sec:intro} claims for the method and precisely the territory OSQP and
Clarabel are not built for. Both are designed for large sparse problems, both
exploit sparsity this experiment does not give them, and both would be expected to
win outright at $n$ in the tens of thousands with a sparse $G$. The table is
evidence that the method is the right choice on its own ground, and no evidence at
all about anyone else's.

\paragraph{So which of the four should a reader use?} The question is fair, and the
table answers it only if one reads the columns against each other, so we state it.
For dense problems whose optima are vertices, any of the four in-family rows returns
the same answer to machine precision, and the choice is about constraints other than
accuracy. Below a few hundred variables the compiled implementations win, and by a
margin no amount of array-language tuning will close, because what is being measured
there is interpreter dispatch. Above that the ordering reverses, and the fast path of
Section~\ref{sec:pdas} leads on every family measured. The two solvers outside the
family are the right choice once $n$ is large and $G$ is sparse, which this
experiment does not test and does not claim to.

DAQP deserves a sharper statement than the ordering gives it, because on three
counts it is simply ahead: it handles semidefinite $G$ through proximal
regularisation, its worst-case iteration count can be certified offline for a
parametric family~\cite{arnstrom2022certification}, and it is library-free C sized
for embedded targets. A reader with a toolchain, a real-time deadline and a
semidefinite Hessian should use it. What the implementation studied here offers
against that is narrower: no compiler and no build step, a
drop-in signature for code already written against the reference, and the arithmetic
advantage above a few hundred variables. Whether that is the better trade is a
property of the reader's situation and not of the method.

\subsection{What these experiments do not establish}
\label{ssec:limits}

These experiments verify agreement between the code and the identities proved
here, not the absence of defects elsewhere, in constraint selection or the
infeasibility verdict or degenerate input, for which a differential comparison
of complete trajectories against an independent implementation would be the
instrument. They say nothing about conditioning: problems use
$G = BB\T + nI$, and the residuals would grow with $\kappa(G)$. That choice also
biases Section~\ref{ssec:pdas}, in the direction that works against the conclusion
drawn there, since a well-conditioned $G$ spreads the minimiser so few bounds bind
where a realistic covariance concentrates it and binds many;
Section~\ref{sec:portfolio} measures how far off that is. And the checks are dense
and of modest size, with single-machine timings.

%% file: sections/s12_portfolio.tex
\section{A problem the synthetic families were not standing in for}
\label{sec:portfolio}

Every experiment so far draws its Hessian as $G = BB\T + nI$ for Gaussian $B$.
Section~\ref{ssec:limits} names that as a limitation and says the bias runs against
the method this note advocates. This section makes the claim concrete rather than
leaving it as a caveat, by solving the problem the method was built for, the
long-only portfolio, on two real return series.

The program is~\eqref{eq:qp2} with the budget as the single equality and
non-negativity as the bounds,
\begin{equation}
  \min_{w}\; \tfrac12 w\T \Sigma w - \rho\,\mu\T w
  \quad\text{subject to}\quad
  \mathbf{1}\T w = 1, \quad w \ge 0,
  \label{eq:portfolio}
\end{equation}
with $\Sigma$ the annualised sample covariance of daily returns and $\mu$ the
sample mean. Two universes, so that nothing below is a property of one market:
the S\&P~500 at $n = \qpPfSpN{}$ assets over $T = \qpPfSpT{}$ days, and the
FTSE~100 at $n = \qpPfFtseN{}$ over $T = \qpPfFtseT{}$.

The FTSE series carries one instrument quoted flat across the whole window. Its
sample variance is zero, so $\Sigma$ is singular, rank $86$ of $87$, and the
solver refuses the problem, correctly: an asset with no variance has no place in a
variance-minimising portfolio. The degenerate column is dropped and
$n = \qpPfFtseN{}$ is the result.

\subsection{Real covariances are the hard case, and in the expected direction}

Table~\ref{tab:portfolio} reports the minimum-variance corner, $\rho = 0$.

\begin{table}[ht]
\centering
\small
\caption{Long-only minimum variance on real data: time per solve and achieved KKT
residual. Sizes, conditioning and active-set size are given per universe.}
\label{tab:portfolio}
\begin{tabular}{lrr}
\toprule
Solver & ms & residual \\
\midrule
\quadprogPortfolioRows
\bottomrule
\end{tabular}
\end{table}

The differences from the synthetic families all run in the same direction.

\paragraph{Worse conditioning, and a far larger active set.} $\kappa(\Sigma)$ is
$\qpPfSpKappa$ on the S\&P~500 against the $O(n)$ of a Gaussian $BB\T + nI$: a
sample covariance is close to low rank plus diagonal, hence ill-conditioned, and
real data looks like that. At the minimum-variance
corner the solver holds $\qpPfSpActive{}$ of $\qpPfSpN{} + 1$ constraints active
and the portfolio holds only $\qpPfSpNames{}$ names. Long-only minimum variance
concentrates onto a vertex at which almost every bound binds, where a
well-conditioned Hessian spreads the optimum instead.

\paragraph{And so the exact walk is far longer.} It takes $\qpPfSpOuter{}$ outer
iterations here against the $\qpPdasBudgetOuter{}$ that the synthetic
budget-plus-bounds family produced at $n = 100$ in Section~\ref{ssec:pdas}. The
iteration count grows with the size of the set the walk must reach, and on real
data that set is large. Every statement in Section~\ref{ssec:pdas} about the
advantage of guessing the active set was therefore measured at the wrong end of
its range: on the S\&P~500 the fast path settles in $\qpPfSpRepairs{}$ repairs
against $\qpPfSpOuter{}$ outer iterations, and is $\qpPfFastSpeed\times$ faster
than the exact walk, against the $\qpCmpFastBudget\times$ of the synthetic family
with the same constraint shape at $n = \qpCmpNhi{}$.

The crossover of Section~\ref{ssec:compare} appears here too, with a point either
side and on real instances constructed for neither purpose: on the S\&P~500 the
exact walk beats both in-family competitors, the C reference and
DAQP~\cite{arnstrom2022daqp}, and on the FTSE~100, at $n = \qpPfFtseN{}$, it loses
to both.

\paragraph{Accuracy survives the conditioning.} The active-set method returns
$\qpPfSpResid$ despite $\kappa(\Sigma) = \qpPfSpKappa$, against five orders of
magnitude more for the interior-point solver at the same requested tolerance, and
eleven on the FTSE~100. Section~\ref{ssec:compare}'s argument reaches its sharpest
form here: the optimum is a vertex with $\qpPfSpActive{}$ active
constraints, which is the worst case for a method that approaches the boundary
asymptotically and the natural case for one that terminates on it.

\subsection{The frontier is what warm starting is for}
\label{ssec:frontier}

A frontier sweep varies $\rho$ in~\eqref{eq:portfolio} and changes nothing else.
That is precisely the family of Section~\ref{sec:sweep}: $\Sigma$, $C$ and $b$
fixed, only the linear term moving. Figure~\ref{fig:portfolio} traces
$\qpPfPoints{}$ points across the S\&P~500 frontier.

\begin{figure}[ht]
\centering
\includegraphics[width=0.86\textwidth]{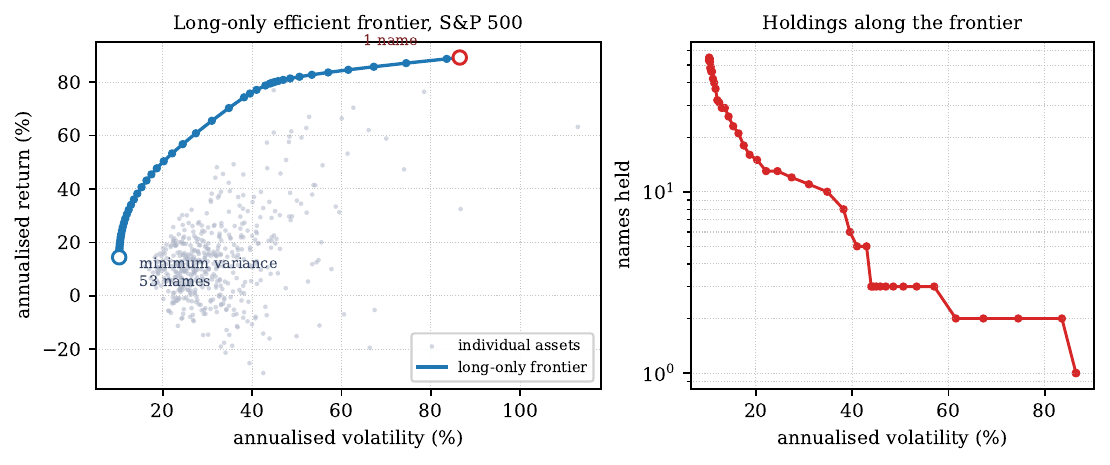}
\caption{Left: the long-only efficient frontier, with the individual assets it
dominates. Right: the number of names held along it, on a log scale, falling from
$\qpPfNamesHi{}$ at the minimum-variance corner to $\qpPfNamesLo{}$ at the
return-seeking end, where the budget concentrates in a single asset.}
\label{fig:portfolio}
\end{figure}

Reusing the factorisation costs $\qpPfWarm$\,ms per point against
$\qpPfCold$\,ms solving each point from scratch, a factor of $\qpPfSweepSpeed$,
and the answers are identical to the cold solves because a reused point is
returned only against the certificate of Section~\ref{sec:sweep}.

The interesting part is where that factor comes from, and it is not where one
would expect. The cached active set stayed optimal on only $\qpPfHits{}$ of the
$\qpPfPoints{}$ points; the other $\qpPfMisses{}$ were misses. A sweep spanning
three decades of $\rho$ moves a long way between consecutive points, the right-hand
panel showing the holdings falling from $\qpPfNamesHi{}$ names to $\qpPfNamesLo{}$
across it, and the active set moves with them. Sampling the
frontier more finely would raise the hit rate directly; nothing here is tuned to
produce one.

So the saving is not mostly the hit path. It is that a \emph{miss} is repaired
rather than restarted. Proposition~\ref{prop:repair} turns a stale set into a
dual-feasible state by dropping the constraints whose multipliers have gone
negative, and the iteration resumes from there instead of from the unconstrained
minimum, keeping $J$ and $R$ throughout. A sweep that misses two points in three
is still $\qpPfSweepSpeed\times$ faster than solving each point cold, which is a
sharper statement about the machinery than a high hit rate would have been: it
says the mechanism that matters is the one that runs when the cache is
\emph{wrong}.

An earlier version of this experiment spaced $\rho$ linearly. Almost every point
then landed on the return-seeking corner, where the solution is a single asset and
does not move, and the cache hit $48$ times out of $50$. That number was real and
any conclusion drawn from it would have been worthless, because the sweep was
solving the same problem forty-nine times instead of tracing a frontier. The
geometric spacing is what makes the left-hand panel a curve, and the honest hit
rate is the one that comes with it.

This section claims that the method's advantages are larger on the problems it was
designed for than on the synthetic families used to test it, and that the direction
of the gap was predicted, not found afterwards. It does not claim this is a
good way to build a portfolio: minimum variance on a raw sample covariance is a
poor estimator, and $\Sigma$ is used here only as a source of realistically shaped
matrices.

%% file: sections/s13_conclusions.tex
\section{Summary}
\label{sec:conclusions}

The Goldfarb--Idnani method is often presented as a sequence of updates, which
makes it look more intricate than it is. Read from the invariants it is short. One
matrix $J$ is carried, defined by $J\T G J = I$ and $J\T A = [R;0]$: its columns
are a $G$-orthonormal basis of $\R^n$, arranged so that the trailing block spans
the working set's feasible subspace. Every quantity the iteration needs is then one
matrix--vector product or one triangular solve away, and every quantity it
\emph{consumes} has a closed form in $(G, A, n_*)$ alone
(Proposition~\ref{prop:invariance}), which is why the representation may be
chosen for speed, and why two implementations holding different factors compute
the same step.

The iteration is one affine path (Proposition~\ref{prop:path}) along which
stationarity holds identically while the entering multiplier rises and the
working-set multipliers fall. Two limits cut it short, and which binds decides
whether a constraint is added or dropped. The objective advances by an exact closed
form (Proposition~\ref{prop:increment}), positive in both step directions, so the
method is a dual ascent and, absent degeneracy, finite
(Proposition~\ref{prop:ascent}). When the path is blocked in both directions the
vector $r$ is a Farkas certificate (Proposition~\ref{prop:farkas}): the
infeasibility verdict is a proof, and the solver already holds it.

Two properties of the classical method are easy to overlook, and the extensions
of Sections~\ref{sec:pdas} and~\ref{sec:sweep} trade one away and exploit the
other. The first is that linear dependence among the working-set normals is
unreachable: the step rules exclude it (Proposition~\ref{prop:rank}), so no rank
test is needed anywhere. Guessing the active set forfeits it, and a primal--dual
active-set fast path therefore admits no finite-termination theorem for general
constraints: there is no $P$-matrix structure to appeal to once the working-set
system stops being a principal submatrix. The second is that the factors depend
on $G$ and the working set but not on the linear term, so a family of programs
differing only in $a$ is solvable from one factorisation, and a stale active set
repairable rather than disposable.

Both extensions are unguaranteed and both are safe, for the same reason: strict
convexity makes the KKT conditions sufficient (Proposition~\ref{prop:sufficient}),
so a candidate can be certified instead of trusted. The strict convexity that
buys it is a choice of scope rather than a boundary of the approach: the dual
method extends to semidefinite $G$, with finite termination, by
Boland~\cite{boland1997}, and by proximal regularisation in the solver
of~\cite{arnstrom2022daqp}. Of the ideas here, that one travels furthest. A method that may fail, but that can recognise its own success,
needs a certificate and not a convergence theory.

One observation generalises past this solver. Below a few hundred variables a
solve costs what its array operations cost to dispatch, not what its
arithmetic costs, so halving the flop count buys nothing and halving the
\emph{iteration} count buys everything: the fast path of Section~\ref{sec:pdas}
gains about as much over the exact walk as rewriting that walk in another language
does, and gains most on the family where the rewrite gains least
(Section~\ref{ssec:compare}). Effort spent on kernels is
wasted where the bottleneck is the interpreter, and effort spent on iteration
counts is wasted where it is not; knowing which regime one is in is the whole of
the decision.

%% file: sections/s14_reproducibility.tex
\section*{Reproducibility}
\label{sec:reproducibility}
\addcontentsline{toc}{section}{Reproducibility}

Every number in this note is produced by a committed script, and no figure or
table was edited by hand. The scripts write the files the note \verb|\input|s, so
a stale figure cannot survive a regeneration unnoticed.

\begin{center}
\small
\begin{tabular}{@{}llll@{}}
\toprule
Script (\texttt{quadprog\_note/}) & Figure & Table & Used in \\
\midrule
\verb|experiment_qp_identities| & --- & --- & below \\
\verb|experiment_qp_pdas| & \verb|quadprog_pdas|
  & \verb|_pdas_defs| & \S\ref{ssec:pdas} \\
\verb|experiment_qp_sweep| & ---
  & \verb|_sweep_defs| & \S\ref{sec:sweep} \\
\verb|experiment_qp_compare| & \verb|quadprog_compare|
  & \verb|quadprog_compare| & \S\ref{ssec:compare} \\
\verb|experiment_qp_portfolio| & \verb|quadprog_portfolio|
  & \verb|quadprog_portfolio| & \S\ref{sec:portfolio} \\
\bottomrule
\end{tabular}
\end{center}

\noindent
Each row's table entry stands for both the tabular fragment and the
\verb|_defs| file of macros beside it, where the script writes both. The identity
script writes no figure, and the sweep script writes one the note does not
include: its cost curves are the natural way to inspect the two regimes
\S\ref{sec:sweep} describes in words, so the figure is left in place for a reader
who wants to see them.

\paragraph{The identities are checked against the code.} A proof establishes the
identities of Sections~\ref{sec:factorisation} to~\ref{sec:sweep} in exact
arithmetic; it does not establish that the implementation satisfies them, which
matters because it reaches them through a packed triangular factor addressed by
hand and through library calls writing in place into a view of $J$'s own buffer.
\verb|experiment_qp_identities| evaluates both sides of all \qpIdentCount{} of them
against the solver's own internal state and not a reimplementation:
\qpIdentChecks{} checks over \qpIdentStates{} factorisation states, worst relative
residual $\qpIdentWorst$. Two come out exactly zero, and one of those is
Remark~\ref{rem:scaleexact}.

\paragraph{Software under study.} \texttt{cvx-quadprog}, version
$0.4.2$~\cite{cvxquadprog}, archived with a DOI. The measurements are tied to that
version rather than to the latest release, so that the reference keeps pointing at
the code that produced them. The experiments import the package's private modules
where the claims are about private state, namely $J$, the packed $R$, and the
vectors $d$, $z$, $r$, and they instrument the shipped routines by wrapping them
instead of reimplementing them, so what is measured is what runs.

\paragraph{Regenerating everything.} From this note's source tree, \verb|make figures|
runs all five scripts and rewrites every generated file; \verb|make compile| then
rebuilds this document. Each script also runs standalone as a module. Setting
\verb|EXPERIMENT_SMOKE=1| shrinks every sweep so that the whole set completes in
seconds, which is how the continuous-integration suite exercises each code path
without reproducing the full runs, and \verb|EXPERIMENT_OUT| redirects the outputs
so that a reduced run cannot overwrite the committed ones.

\paragraph{Determinism.} Every random problem is drawn from a seed derived
arithmetically from the family, size and instance index, so the tables reproduce
across runs and machines. Timings do not, and are not claimed to: they were taken
on one machine against NumPy built on Apple Accelerate, with no thread count set.
That is not incidental. This solver pushes its work into BLAS calls, so its
timings inherit whatever threading decision the installed library makes by default,
and those defaults differ between libraries by more than the effects measured here.
Accelerate exposes no thread control, so there was nothing to set; on a library
that does, the same experiment can read differently. The residuals of
the identity check below and the counts of \S\ref{ssec:pdas} are machine-independent
to the last digit or two of rounding, and are what the note rests on.

\paragraph{Optional dependency.} The comparison in \S\ref{ssec:compare} includes
the reference C implementation where a wheel for it is installed; building it needs
the C toolchain the package under study exists to avoid, so that row is omitted,
with a printed note, on a host that lacks it. Every other row is unconditional.